\documentclass[a4paper,11pt,reqno]{amsart}
\usepackage[left=3.0cm,right=3.0cm,bottom=3.0cm]{geometry} 
\usepackage{amssymb}
\usepackage{amsmath}
\usepackage{mathrsfs}
\usepackage[usenames]{color}
\usepackage{bm}
\usepackage{enumerate}

\numberwithin{equation}{section}

\allowdisplaybreaks

\begin{document}

\newtheorem{theorem}{Theorem}[section] 
\newtheorem{proposition}[theorem]{Proposition}
\newtheorem{corollary}[theorem]{Corollary}
\newtheorem{lemma}[theorem]{Lemma}

\theoremstyle{definition}
\newtheorem{assumption}[theorem]{Assumption}
\newtheorem{definition}[theorem]{Definition}

\theoremstyle{definition} 
\newtheorem{remark}[theorem]{Remark}
\newtheorem{remarks}[theorem]{Remarks}
\newtheorem{example}[theorem]{Example}
\newtheorem{examples}[theorem]{Examples}
\newenvironment{pf}%
{\begin{sloppypar}\noindent{\bf Proof.}}%
{\hspace*{\fill}$\square$\vspace{6mm}\end{sloppypar}}
\def\mE{{\mathbb E}}
\def\mF{{\mathbb F}}
\def\mG{{\mathbb G}}
\def\mL{{\mathbb L}}
\def\mX{{\mathbb X}}
\def\mP{{\mathbb P}}
\def\R{{\mathbb R}}
\def\N{{\mathbb N}}
\def\C{{\mathbb C}}
\def\Q{{\mathbb Q}}
\def\Z{{\mathbb Z}}
\def\mH{\mathbb H}
\def\mA{\mathbb A}
\def\mT{\mathbb T}
\def\cD{{\mathcal D}}
\def\cB{{\mathcal B}}
\def\cE{{\mathcal E}}
\def\cF{{\mathcal F}}
\def\cA{{\mathcal A}}
\def\cH{{\mathcal H}}
\def\cS{{\mathcal S}}
\def\cT{{\mathcal T}}
\def\cP{{\mathcal P}}
\def\cL{{\mathcal L}}
\def\cK{{\mathcal K}}
\def\cJ{{\mathcal J}}
\def\cR{{\mathcal R}}
\def\cM{{\mathcal M}}
\def\bH{{\bf H}}
\def\bT{{\bf T}}
\def\tT{{\widetilde{T}}}
\def\im{{\mathrm{im}}}

\def\vp{{\varphi}}
\def\eps{\varepsilon}
\def\hA{\widehat{A}}

\def\supp{{\mathrm{supp}}}
\def\esssup{{\mathrm{ess\,sup}}}
\def\Re{{\mathrm{Re}}}
\def\Im{{\mathrm{Im}}}
\def\hperp{{^{_\perp}}}
\def\hW{\widehat{W}}
\newcommand{\essinf}[1]{{\mathrm{ess}}\!\inf_{\!\!\!\!\!\!\!\!\! #1}}
\def\div{{\mathrm {div\,}}}

\def\hookd{\stackrel{_d}{\hookrightarrow}}
\def\hook{{\hookrightarrow}}
\def\la{{\langle}}
\def\lla{\left\langle}
\def\ra{{\rangle}}
\def\rra{\right\rangle}
\def\sL{{\mathscr L}}
\def\sLis{{\mathscr L}_{is}}
\def\vp{{\varphi}}

\newcommand{\trace}[1]{{\langle {#1} \rangle}_\bullet}
\newcommand{\jump} [1]{{\langle \! \langle {#1} \rangle \! \rangle}_\bullet}

\hyphenation{Lipschitz}

\newcommand{\narrowarray}{\setlength{\arraycolsep}{0.2em}}
\newcommand{\normalarray}{\setlength{\arraycolsep}{0.4em}}
\newcommand{\widearray}  {\setlength{\arraycolsep}{0.6em}}
\narrowarray

\sloppy
\title
	[Stokes Equations on a Wedge Subject to Navier Boundary Conditions]
	{Optimal regularity \\ for the Stokes Equations on a 2D Wedge Domain \\ Subject to Navier Boundary Conditions}


\author[M.\ K\"ohne]{Matthias K\"ohne}
\address{Heinrich-Heine-Uni\-ver\-sit\"at D\"usseldorf,
	Mathematisch-Naturwissenschaftliche Fakult\"at,
	Mathematisches Institut}
\email{matthias.koehne@hhu.de}

\author[J. Saal]{J\"urgen Saal}
\address{Heinrich-Heine-Uni\-ver\-sit\"at D\"usseldorf,
	Mathematisch-Naturwissenschaftliche Fakult\"at,
	Mathematisches Institut}
\email{juergen.saal@hhu.de}

\author[L.\ Westermann]{Laura Westermann}
\address{Heinrich-Heine-Uni\-ver\-sit\"at D\"usseldorf,
	Mathematisch-Naturwissenschaftliche Fakult\"at,
	Mathematisches Institut}
\email{laura.westermann@hhu.de}

\thispagestyle{empty}
\parskip0.5ex plus 0.5ex minus 0.5ex

\begin{abstract}
We consider the Stokes equations subject to Navier boundary conditions on a two-dimensional wedge domain
with opening angle $\theta_0 \in (0,\,\pi)$.
We prove existence and uniqueness of solutions with optimal regularity in an $L^p$-setting.
The results are based on optimal regularity results for the Stokes equations subject to perfect slip boundary
conditions on a two-dimensional wedge domain that have been obtained by the authors in \cite{Koehne-Saal-Westermann:Stokes-Wedge}.
Based on a detailed study of the corresponding trace operator on anisotropic Sobolev-Slobodeckij type function spaces
on a two-dimensional wedge domain we are able to generalize the results proved in \cite{Koehne-Saal-Westermann:Stokes-Wedge}
to the case of inhomogeneous boundary conditions.
Existence and uniqueness of solutions to the Stokes equations subject to (inhomogeneous) Navier boundary conditions
are then obtained using a perturbation argument.
\end{abstract}
\maketitle
{\footnotesize
}

{\renewcommand\thefootnote{}\footnote{\today}\addtocounter{footnote}{-1}}


\section{Introduction and Main Result}
\label{sec_intro}

The main objective of this note is to study the (instationary) Stokes equations subject to (inhomogeneous) Navier boundary conditions
\begin{equation}
	 \label{problem_navierslipbc}
	 	\begin{array}{r@{\ =\ }lll}
	 		 \partial_t u - \Delta u + \nabla p &  f &
			\quad \text{in} & J \times G, \\[0.25em]
			\div u & g & \quad \text{in} & J \times G, \\[0.25em]
			\alpha u \cdot \tau - \tau^T D_\pm (u) \nu & h_1& \quad \text{on} & J \times \Gamma, \\[0.25em]
	 		u \cdot \nu & h_0 & \quad \text{on} & J \times \Gamma,\\[0.25em]
			u(0) & u_0 & \quad \text{in} & G
	 	\end{array}
\end{equation}
on a two-dimensional wedge domain $G$.
We aim at existence and uniqueness of solutions with optimal regularity in an $L^p$-setting for $p \in (1,\infty)$.
The wedge domain is defined as
\begin{equation}
	\label{def_wedge}
	G:= \Big\{ \, (x_1, x_2) = (r \cos \theta, r \sin \theta): \ r > 0,\ 0 < \theta < \theta_0 \, \Big\} \subseteq \R^2
\end{equation}
with opening angle $ \theta_0 \in (0, \pi)$ and $J=(0, T)$ with $T>0$.
Here $\alpha$ is a given (variable) parameter, $\nu$ and $\tau$ denote the unit outer normal vector and a unit tangential vector on $\Gamma := \partial G \setminus \{0 \}$ respectively.
We have $\nu_1 = - e_2$ and set $\tau_1 := -e_1$ as the unit outer normal vector and a unit tangential vector on $\Gamma_1 := (- \infty,0) \cdot \tau_1$.
Furthermore, we have $\nu_2 = (-\sin \theta_0, \cos \theta_0)^T$ and set $\tau_2 := (\cos \theta_0, \sin \theta_0)^T$ as the unit normal vector and a unit tangential vector on $\Gamma_2 := (0, \infty) \cdot \tau_2$.
Thus, the boundary of $G$ is decomposed as
\begin{equation}\label{def_boundary}
\Gamma = \Gamma_1 \cup \Gamma_2 = \partial G \setminus \{ 0\}
\end{equation}
and we have $(\tau,\,\nu) = (\tau_j,\,\nu_j)$ on $\Gamma_j$ for $j = 1,\,2$.
Note that $(\tau_j,\,\nu_j)$ is positively oriented for $j = 1,\,2$.
The boundary conditions in the third and fourth equation of system \eqref{problem_navierslipbc} have to be understood as
\begin{align*}
	\alpha u \cdot \tau_1 - \tau^T_1 D_\pm(u) \nu_1 &= h^{(1)}_1 \ \text{on} \ J \times \Gamma_1, \\
	\alpha u \cdot \tau_2 - \tau^T_2 D_\pm(u) \nu_2 &= h^{(2)}_1 \ \text{on} \ J \times \Gamma_2, \\ 
	u \cdot \nu_1 &= h^{(1)}_0 \ \text{on} \ J \times \Gamma_1, \\
	u \cdot \nu_2 &= h^{(2)}_0 \ \text{on} \ J \times \Gamma_2,
\end{align*}
where $h^{(j)}_\ell := h_\ell|_{\Gamma_j}$ for $\ell = 0,\,1$ and $j = 1,\,2$.
Moreover, $ D_\pm (u):= \frac{1}{2} (\nabla u \pm \nabla u^T)$ denote the rate of deformation tensor and the rate of rotation tensor, respectively.

If $\psi: G \longrightarrow \R$ or $\psi: \Gamma \longrightarrow \R$ is a function, we denote
by $\langle \psi \rangle_j := \lim_{x \rightarrow 0} \psi|_{\Gamma_j}(x)$ its trace at the corner $x = 0$ of the wedge $G$
taken w.\,r.\,t.\ its values on $\Gamma_j$ for $j = 1,\,2$, whenever it exists.
By $\langle\!\langle \psi \rangle\!\rangle_\bullet := \langle \psi \rangle_2 - \langle \psi \rangle_1$ we denote
its {\itshape jump} across the corner, whenever the two traces exist.
Finally, we denote by $\trace{\psi} := \langle \psi \rangle_1 = \langle \psi \rangle_2$
its unique trace at the corner, provided that $\jump{\psi} = 0$.
Thus, a condition like $\trace{\psi} = 0$ implicitly requires $\jump{\psi} = 0$.

We aim at solutions
\begin{equation}
	\label{solution_class}
	(u,\,p) \in \mE := \mE_u \times \mE_p,
\end{equation}
where
\begin{align*}
\mE_u&:= W^{1,p}(J, L^p(G, \R^2))\cap L^p(J, W^{2,p}(G, \R^2))\\
\mE_p&:=L^p(J, \widehat{W}^{1,p}(G))
\end{align*}
are given as anisotropic (homogeneous) Sobolev spaces; see~Section~\ref{sec_not}.
Of course, in this setting uniqueness of the pressure $p$ has to be understood as uniqueness up to an additive constant.
Then, necessarily, the given data in \eqref{problem_navierslipbc} have to satisfy the regularity conditions
\begin{equation*}
	\begin{array}{rclcl}
		  f & \in & \mF_f          & := & L^p(J, L^p(G, \R^2)), \\[0.25em]
		  g & \in & \mF_g          & := & W^{1/2}_p(J, L^p(G))\cap L^p(J, W^{1, p}(G)), \\[0.25em]
		h_1 & \in & \mF_\tau       & := & \{\,h: \Gamma \longrightarrow \R\,:\,h|_{\Gamma_j} \in \mF^{(j)}_\tau\ \textrm{for}\ j = 1,\,2\,\},\ \textrm{where} \\[0.25em]
		    &     & \mF^{(j)}_\tau & := & W^{1/2 - 1/2p}_p(J, L^p(\Gamma_j)) \cap L^p(J, W^{1 - 1/p}_p(\Gamma_j)),\ j = 1,\,2, \\[0.25em]
		h_0 & \in & \mF_\nu        & := & \{\,h: \Gamma \longrightarrow \R\,:\,h|_{\Gamma_j} \in \mF^{(j)}_\nu\ \textrm{for}\ j = 1,\,2\,\},\ \textrm{where} \\[0.25em]
		    &     & \mF^{(j)}_\nu  & := & W^{1 - 1/2p}_p(J, L^p(\Gamma_j)) \cap L^p(J, W^{2-1/p}_p(\Gamma_j)),\ j = 1,\,2, \\[0.25em]
		u_0 & \in & \mF_0          & := & W^{2 - 2/p}_p(G, \R^2),
	\end{array}
\end{equation*}
i.\,e.\ we have to work with anisotropic Sobolev-Slobodeckij spaces; see~Section~\ref{sec_not}.
For convenience we abbreviate
\begin{equation}
	\label{data_class}
	\mF := \mF_f \times \mF_g \times \mF_\tau \times \mF_\nu \times \mF_0.
\end{equation}
We employ the space $BUC^1(\Gamma) := \{\,\alpha: \Gamma \longrightarrow \R\,:\,\alpha|_{\Gamma_j} \in BUC^1(\Gamma_j),\ j = 1,\,2\,\}$ for the coefficients.
Besides the obvious necessary compatibility conditions between the right-hand side $g$ in the divergence equation and the initial datum $u_0$
and between the boundary datum $h_j$ and the initial datum $u_0$, respectively,
there is a somewhat hidden but well-known necessary compatibility condition between $g$ and the normal boundary datum $h_0$.
To formulate this compatibility condition we denote by $p' \in (1,\infty)$ the dual exponent of $p \in (1,\infty)$
and define the functional $F(\gamma, \eta): W^{1, p'}(G) \longrightarrow \R$ for $\gamma \in \mF_g$ and $\eta \in \mF_\nu$ as
\begin{equation}
	\label{funktional}
	[F(\gamma, \eta)](\phi) := \left(\eta,\,\phi\right)_{\Gamma} - (\gamma,\,\phi)_G, \qquad \phi \in W^{1,p'}(G).
\end{equation}
Since
\begin{equation*}
	\begin{array}{rcl}
		[F(g, h_0)](\phi) & = & \left(h_0,\,\phi\right)_{\Gamma} - (g,\,\phi)_G = (u \cdot \nu,\,\phi)_{\Gamma} - (\div u,\,\phi)_G \\[0.75em]
			& = & (u,\,\nabla \phi)_G \in W^{1,p}(J), \qquad \phi \in W^{1, p'}(G),
	\end{array}
\end{equation*}
we infer that
\begin{equation*}
	F(g, h_0) \in W^{1,p}\big(J, (W^{1,p'}(G), \|\nabla \cdot\|_{L^{p'}(G, \R^2)})'\big).
\end{equation*}
By the fact that $C_c^\infty(\bar{G})$ is dense in $\widehat{W}^{1,p'}(G)$
it follows that $F(g, h_0)  \in W^{1,p}(J,\widehat{W}^{-1,p}(G))$;
cf.\ Corollary~\ref{Homogeneous-Density-Wedge}.

\begin{remark}\label{rem_compat}
For $g \in \mF_g$ the requirement $F(g, 0) \in W^{1,p}(J,\widehat{W}^{-1,p}(G))$
is equivalent to $g \in W^{1,p}(J,\widehat{W}^{-1,p}(G))$,
while for $h_0 \in \mF_\nu$ the requirement $F(0, h_0) \in W^{1,p}(J,\widehat{W}^{-1,p}(G))$
is equivalent to $h_0|_{\Gamma_j} \in W^{1, p}(J, \widehat{W}^{-1 / p}_p(\Gamma_j))$ for $j = 1,\,2$.
\end{remark}
\noindent
Now, our main result reads as follows.

\begin{theorem}\label{thm_navierslip}
Let $J = (0,T)$ with $0 < T < \infty$ and let $G \subset \R^2$ be defined as in \eqref{def_wedge} with $\theta_0 \in (0,\,\pi)$.
Let $p \in (1, \infty) \setminus \{ \frac{2 \theta_0}{3 \theta_0- \pi},\ \frac{2 \theta_0}{3 \theta_0 - 2\pi},\ \frac{3}{2},\ 2,\ 3\, \}$.
Let $\alpha \in BUC^1(\Gamma)$ with $\trace{\alpha} = 0$.
Suppose the data satisfy the regularity condition
\begin{equation*}
(f,\ g,\ h_1,\ h_0,\ u_0) \in \mF
\end{equation*}
and the compatibility conditions
\begin{equation*}
	\begin{array}{rcll}
		\mbox{div}\,u_0 & = & g|_{t=0}, & \quad \textrm{if}\ p > 2 ,\\[0.5em]
		u_0 \cdot \nu & = & h_0|_{t=0}, & \quad \textrm{if}\ p> \frac{3}{2}, \\[0.5em]
		\alpha u_0 \cdot \tau - \tau^T D_{\pm} (u_0) \nu & = & h_1|_{t=0}, & \quad \textrm{if}\ p> 3,
	\end{array}
\end{equation*}
as well as
\begin{equation*}
	F(g, h_0) \in W^{1,p}(J, \widehat{W}^{-1,p}(G)).
\end{equation*}
If the boundary condition is posed based on $D_+$,
then assume the compatibility conditions $\jump{\partial_\tau h_0 + h_1} = 0$ in $J$, if $p > 2$, and
\begin{equation*}
	{\textstyle \frac{1}{2}} \langle \partial_{\tau_1} h_0 \rangle_1 + {\textstyle \frac{1}{2}} \langle \partial_{\tau_2} h_0 \rangle_2 = \trace{\partial_\tau h_0 + h_1} \quad \textrm{in}\ J,
		\qquad \textrm{if}\ \theta_0 = {\textstyle \frac{\pi}{2}}\ \textrm{and}\ p > 2.
\end{equation*}
If the boundary condition is posed based on $D_-$,
then assume the compatibility conditions $\jump{h_1} = 0$ in $J$, if $p > 2$, and
\begin{equation*}
	- {\textstyle \frac{1}{2}} \langle \partial_{\tau_1} h_0 \rangle_1 - {\textstyle \frac{1}{2}} \langle \partial_{\tau_2} h_0 \rangle_2 = \trace{h_1} \quad \textrm{in}\ J,
		\qquad \textrm{if}\ \theta_0 = {\textstyle \frac{\pi}{2}}\ \textrm{and}\ p > 2.
\end{equation*}
Then there exists a unique solution $(u,p) \in \mE$ to \eqref{problem_navierslipbc}.
\end{theorem}

\begin{remark}
The values $p = 2$, $p = \frac{2\theta_0}{3\theta_0 - \pi}$ and $p = \frac{2 \theta_0}{3 \theta_0- 2\pi}$ with $\theta_0 \in (0, \pi)$ are excluded in Theorem~\ref{thm_navierslip} due to technical reasons.
In Section \ref{sectionlaplaceneumann} we solve the Laplace equation subject to Neumann boundary conditions on the wedge domain by transforming this problem into a problem on a layer domain.
The latter is then solved using the operator sum method, which is based on the Kalton-Weis theorem.
Due to this method a condition on the spectrum of the operators appears, which excludes $p = \frac{2\theta_0}{3\theta_0 - \pi}$ and $p = \frac{2 \theta_0}{3 \theta_0- 2 \pi}$.
Moreover, the transformation from the layer back to the wedge introduces weights.
The norms in the corresponding weighted function spaces can be estimated thanks to Hardy's inequality for all $p \in (1, \infty)$ except for $p=2$.
See Lemma \ref{lemma_hardy} for Hardy's inequality on the wedge.
Thanks to the solvability of the Laplace equation we can then prove the solvability of equation \eqref{probu2perfectslip} below,
which is a crucial step for the proof of Theorem~\ref{thm_navierslip}.
The values $p = \frac{3}{2}$ and $p = 3$ are excluded in Theorem~\ref{thm_navierslip} to allow for inhomogeneous right-hand sides,
which leads to the necessity of the stated compatibility conditions.
These two values could be included for the case $g = 0$ and $h_0 = h_1 = 0$.
\end{remark}

Note that literature on the Stokes equations on domains with conical
boundary points is still somewhat rare, in particular for the instationary
system. Results on the stationary Stokes system and on 
pure Dirichlet conditions can be found, e.g., in
\cite{Kondratev:Elliptic-Equations-Conical-Angular-Points, Kellogg-Osborn:Stokes-Convex-Polygon,Dauge:Stationary-Stokes-With-Corners, Mazya-Rossmann:Elliptic-Equations,
Grisvard:Elliptic-Problems, Deuring:Stokes-With-Conical-Boundary-Points, Guo-Schwab:Stokes-Polygonal-Domains, Kozlov-Rossmann:Nonstationary-Stokes-Cone}.
For a Lipschitz approach to Robin type boundary conditions we refer to \cite{Monniaux-Shen:Stokes-Irregular-Domains}. For more ``classical''
approaches to Navier boundary conditions on smooth domains we refer to
\cite{Shibata-Shimada:Stokes-Robin-Resolvent-Estimate, Farwig-Rosteck:Stokes-Navier-Boundary-Conditions, Hobus-Saal:Partial-Slip}, 
just to name a few. For an overview on the Stokes equations
including approaches to various boundary conditions and to non-smooth
domains see \cite{Hieber-Saal:Stokes-Equation}.
This note essentially relies on the approach to perfect slip conditions presented in \cite{Koehne-Saal-Westermann:Stokes-Wedge, Maier-Saal:Stokes-Wedge}.

To provide an outline for the following sections
we summarize the strategy of the proof of Theorem~\ref{thm_navierslip}.
At the end, problem \eqref{problem_navierslipbc} is a perturbed variant of the problem
\begin{equation}\label{probu2perfectslip}
	 	\begin{array}{r@{\ =\ }lll}
	 		 \partial_t u - \Delta u + \nabla p & f &
			\quad \text{in} & J \times G, \\[0.25em]
			\div u & g & \quad \text{in } & J \times G, \\[0.25em]
			- \tau^T D_\pm(u) \nu & h_1 & \quad \text{on} & J \times \Gamma ,\\[0.25em]
	 		u \cdot \nu & h_0 & \quad \text{on} & J \times \Gamma,\\[0.25em]
			u(0) & u_0& \quad \text{in} & G
	 	\end{array}
\end{equation}
with fully inhomogeneous right-hand sides $(f, g, h_1, h_0, u_0) \in \mF$.
Therefore, it is sufficient to show existence and uniqueness of solutions $(u, p) \in \mE$
to problem \eqref{probu2perfectslip}, provided the data satisfy appropriate compatibility conditions.
This is achieved by Corollary~\ref{cor_free_perfect_slip}.

This result, in turn, relies on the unique solvability of the Stokes equations
subject to inhomogeneous perfect slip boundary conditions
\begin{equation*}
	u \cdot \nu = h_0, \qquad \textrm{curl}\,u = h_1 \qquad \textrm{on}\ J \times \Gamma.
\end{equation*}
The latter problem is dealt with in Theorem~\ref{thm_perfect_slip}.
On the one hand, the proof of Theorem~\ref{thm_perfect_slip} relies on the result \cite[Corollary~1]{Koehne-Saal-Westermann:Stokes-Wedge},
which provides optimal regularity for the Stokes equations subject to homogeneous perfect slip boundary conditions in the $L^p$-setting for all $p\in (1, \infty)$.
On the other hand, to cope with the inhomogeneous boundary conditions,
for the proof of Theorem~\ref{thm_perfect_slip} we also need to show optimal regularity for the Laplace equation subject to Neumann boundary conditions
in the space $\widehat{W}^{1,p}(G)$ for all $p \in (1, \infty) \setminus \{ \frac{2 \theta_0}{3 \theta_0- \pi},\ \frac{2 \theta_0}{3 \theta_0 - 2 \pi},\ 2\, \}$.
This is accomplished by Corollary \ref{cor_probl_zeit},
where we show the invertibility of the operator $A_{L,T} \phi:= \Delta \phi$ associated to the problem
\begin{equation}
	 	\begin{array}{r@{\ =\ }lll}
	 		 \Delta \phi &  f  & \quad \text{in} & J \times G, \\[0.25em]
	 		  \partial_\nu \phi  & 0 & \quad \text{on} & J \times \Gamma
	 	\end{array}
\end{equation}
to obtain $ \phi \in L^p(J,K^3_p(G))$ for $f \in L^p(J,\widehat{W}^{1,p}(G))$.
For a definition of the weighted Sobolev space $K^3_p(G)$ see \eqref{def_weighted_sobolevspace} below.

Now, this note is organized as follows.
In Section~\ref{sec_not} we introduce the notation.
Section~\ref{sectionlaplaceneumann} is devoted to the proof of Corollary \ref{cor_probl_zeit},
i.\,e.\ to the treatment of the Laplace equation subject to Neumann boundary conditions in a wedge within the above function spaces.
Finally, in Section~\ref{sec_navierslipbc} we prove the unique solvability of problem \eqref{probu2perfectslip}
and we provide a complete proof of Theorem \ref{thm_navierslip}.
As auxiliary results, we provide several generic trace theorems for the wedge domain $G$ for anisotropic Sobolev-Slobodeckij spaces,
which may be of independent interest.
For convenience this note is complemented by two appendices,
where we discuss Hardy's inequality for the wedge domain $G$
and approximation results in homogeneous Sobolev spaces $\widehat{W}^m_p$ for a large class of unbounded domains.

\section{Notation}\label{sec_not}

Let $X$ be a Banach space, let $1 \leq p \leq \infty$ and let $\Omega\subset \R^2$ be a domain.
We set $C_c^\infty(\Omega, X):= \{u \in C^\infty(\Omega,X): \ \text{supp}(u) \subset\!\!\!\subset \Omega \} $ where $\text{supp}(u)$ is the support of $u$.
We denote by $L^p(\Omega, X)$ the $X$-valued Bochner-Lebesgue space
and we define $W^{k,p}(\Omega, X)$ to be the Sobolev space of order $k \in \N$ and we set $W^{0,p}(\Omega, X):=L^p(\Omega, X)$.
We denote by $W_0^{k,p}(\Omega, X)$ the closure of $C_c^\infty(\Omega, X)$ in the space $W^{k,p}(\Omega, X)$.
Furthermore, for $s=k + \lambda$ with $k \in \N_0$ and $0<\lambda<1$ we define $W^{s}_p(\Omega, X)$ to be the Sobolev-Slobodeckij space
that consists of all functions $u \in W^{k,p}(\Omega, X)$ satisfying
\[ \|u \|_{W^{s}_p(\Omega, X)}:= \| u\|_{W^{k,p}(\Omega, X)} + \sum_{|\alpha |= k} \left( \int_\Omega \int_\Omega \frac{\|\partial^\alpha u(y)- \partial^\alpha u(x)\|^p_X}{|y-x|^{n+ \lambda p}} dy dx \right)^{1/p} < \infty.
\]
For $k \in \N_0$ the homogeneous Sobolev space of scalar valued functions is defined as
\begin{equation*}
	\widehat{W}^{k,p}(\Omega, X):= \{\,u \in L^1_{loc}(\Omega, X): \ \partial^\alpha u \in L^p(\Omega, X), \ |\alpha|=k\,\},
\end{equation*}
and equipped with the seminorm
\begin{equation*}
	\|u\|_{\widehat{W}^{k,p}(\Omega, X)}:= \sum_{|\alpha|=k} \|\partial^\alpha u\|_{L^p(\Omega, X)}.
\end{equation*}
Of course, for $X = \R$ we write $L^p(\Omega)$, $W^{k, p}(\Omega)$, $W^{s}_p(\Omega)$, \dots
for $L^p(\Omega, X)$, $W^{k, p}(\Omega, X)$, $W^{s}_p(\Omega, X)$, \dots, respectively.

Now, let $G \subset \R^2$ be the wedge domain defined in \eqref{def_wedge} with opening angle $\theta_0 \in (0,\,\pi)$.
We define the Kondrat'ev spaces as
\begin{equation*}
	L^p_\gamma(G):= L^p(G, \rho^\gamma d(x_1,x_2)), \quad \rho:=|(x_1,x_2)|, \ \gamma \in \R,
\end{equation*}
and for $m \in \N_0$ as
\begin{equation} \label{def_weighted_sobolevspace}
	K^m_{p,\gamma}(G):= \{\,u \in L^1_{loc}(G): \ \rho^{|\alpha| -m} \partial^\alpha u\in L^p_\gamma(G), \ |\alpha| \leq m\,\}, \quad \gamma \in \R.
\end{equation}
The space $K^m_{p, \gamma}(G)$ equipped with the norm
\begin{equation*}
	 \|u \|_{K^m_{p, \gamma}(G)}:= \left(\sum_{|\alpha| \leq m} \|\rho^{|\alpha| -m} \partial^\alpha u \|^p_{L^p_{\gamma}(G)} \right)^{1/p}
\end{equation*}
is a Banach space for all $m \in \N_0$ and all $\gamma \in \R$ and we abbreviate $K^m_{p}(G):= K^m_{p,0}(G)$.
For $k \in \N$ the weighted homogeneous Sobolev space is defined as
\begin{equation*}
	\widehat{W}_{\gamma}^{k,p}(G):= \{\,u \in L^1_{loc}(G): \ \partial^\alpha u \in L^p_{\gamma}(G), \ |\alpha|=k\,\}, \quad \gamma \in \R,
\end{equation*} 
and equipped with the seminorm
\begin{equation*}
	\|u\|_{\widehat{W}_\gamma^{k,p}(G)}:= \|u\|_{\widehat{W}_\gamma^{k,p}}:= \sum_{|\alpha|=k} \|\partial^\alpha u\|_{L^p_\gamma(G)} 
\end{equation*}
for $k \in \N$ and $\gamma \in \R$.

The norm on a generic Banach space $X$ is usually denoted by $\|\cdot\|_X$.
If $Y$ is another Banach space,
then $\sL(X,Y)$ denotes the space of all continuous, linear operators from $X$ to $Y$
and $\sLis(X,Y)$ denotes the subspace of all linear isomorphisms from $X$ onto $Y$.
For $Y = X$ we employ the abbreviations $\sL(X)$ and $\sLis(X)$, respectively. 
\pagebreak

If $\psi: \Gamma \longrightarrow \R$ is a function, we occasionally denote by $\psi^{(j)} = \psi|_{\Gamma_j}$
the restriction to $\Gamma_j$ for $j = 1,\,2$.
The same notation is also occasionally used for vector fields $\psi: \Gamma \longrightarrow \R^m$ with $m \in \N$
and should not be confused with the components of $\psi$ in this case.
Moreover, if $\psi: \Gamma_j \longrightarrow \R$ with $j \in \{\,1,\,2\,\}$ is a function that is defined
on one of the smooth parts of the boundary of $G$ only,
we also employ the notation $\langle \psi \rangle_j := \lim_{x \rightarrow 0} \psi(x)$ for its trace at the corner $x = 0$ of the wedge $G$.



\section{The Laplace Equation subject to Neumann Boundary Conditions}\label{sectionlaplaceneumann}
Let $G \subset \R^2$ be the wedge domain defined as in \eqref{def_wedge} and $J=(0,T)$ with $0 < T < \infty$. The objective of this section is to consider the problem
\begin{equation*} 
	 	\begin{array}{r@{\ =\ }lll}
	 		 \Delta \phi &  f  &\quad \text{in} & J \times G, \\[0.5em]
	 		  \partial_\nu \phi  & 0& \quad \text{on} & J \times \Gamma
	 	\end{array}
\end{equation*}
and to show its optimal regularity.
Here $\nu$ denotes the unit outer normal vector at $\Gamma$ with $\Gamma$ defined as in \eqref{def_boundary}.
Recall that $\tau_1= -e_1$ and $\nu_1= - e_2$ on $\Gamma_1= (- \infty,0) \cdot \tau_1$ and
$\tau_2= (\cos \theta_0, \sin \theta_0)^T$ and $\nu_2= (-\sin \theta_0, \cos \theta_0)^T $ on $\Gamma_2= (0, \infty) \cdot \tau_2$, respectively.
The boundary condition in the above system is to be understood as
\begin{align*}
\partial_{\nu_1 } \phi &= 0 \quad \text{on} \ \Gamma_1, \\
\partial_{\nu_2 } \phi &= 0 \quad \text{on} \ \Gamma_2. 
\end{align*}
Here, optimal regularity of the Neumann-Laplace equation means to show the invertibility of the operator $A_{L,T} \phi := \Delta \phi$, where
$$A_{L,T}: \ L^p(J, K^3_p(G)) \rightarrow L^p(J, \widehat{W}^{1,p}(G)) $$
for all $p \in (1, \infty) \setminus \{ \frac{2 \theta_0}{3 \theta_0- \pi},\ \frac{2 \theta_0}{ 3 \theta_0 - 2 \pi}, 2\, \}$. 

The strategy will be to start by considering the time independent Neumann problem for the Laplace operator
\begin{equation} \label{probl_laplace_neumann}
	 	\begin{array}{r@{\ =\ }lll}
	 		 \Delta \phi &  f  &\text{in} &   G, \\[0.5em]
	 		  \partial_\nu \phi  & 0& \text{on} &  \Gamma
	 	\end{array}
\end{equation}
and to transform it onto a layer domain $\Omega:=\R \times (0, \theta_0)$ in a first step. Using the operator sum method we can then show the well-posedness of the transformed problem in the unweighted $L^p$-setting. In a second step we will show higher regularity of the transformed problem and then transform it back onto the wedge domain.

\begin{remark}
In \cite[Chapter 4]{Grisvard:Elliptic-Problems} the Laplace equation subject to general boundary conditions, where the Neumann boundary conditions are included, is studied on polygonal domains.
There, localizing the vertices and transforming the Laplace equation to a layer domain yields the same form of the Laplace equation on the layer as in our setting.
Hence, alternatively to the operator sum method, by modifying a step in the proof of \cite[Theorem~4.3.2.3]{Grisvard:Elliptic-Problems} we could also prove the invertibility of the transformed Neumann-Laplace operator on the layer.
For this approach a suitable variant of the condition (4.3.2.10) in \cite{Grisvard:Elliptic-Problems} has to be satisfied, which leads to a constraint on the parameter $p$ of the $L^p$-space.
Now, inserting into that equation $ \beta_p= 3-\frac{2+\gamma}{p}$ instead of $\frac{2}{p'}=2-\frac{2}{p}$, we get a condition that is equivalent to our spectral condition \eqref{cond_eigenv}; see also Remark \ref{remark_spectral_condition}.
Thus, this approach would lead to optimal regularity for the Neumann problem for the Laplace operator for the same values of $p$.
However, we prefer to provide a self-contained proof based on the operator sum method.
\end{remark}

Let's start with the transformation of problem \eqref{probl_laplace_neumann} onto the layer domain.
We set $\Omega:= \R \times I$, with $I:=(0, \theta_0)$ where $\theta_0$ is the angle of the wedge $G$.
We write the inverse of the transformation to polar coordinates as
$$\psi_P: \ \R_+ \times I \rightarrow G, \quad (r, \theta ) \mapsto (r \cos \theta, r \sin \theta) = (x_1, x_2). $$
We use the Euler transformation $r=e^x$ in radial direction and write, by a slight abuse of notation, $x \in \R$ for the new variable.
We set
$$\psi_E: \ \Omega \rightarrow \R_+ \times I, \quad (x, \theta) \mapsto (e^x, \theta)=: (r, \theta).  $$
It is not difficult to see that
$$ \psi:= \psi_P \circ \psi_E : \Omega \rightarrow G$$
is a diffeomorphism. We set
$$\Psi \phi := \phi \circ \psi \quad \text{and} \quad \Psi^{-1} \varphi := \varphi \circ \psi^{-1}. $$
Analogously to \cite{Maier-Saal:Stokes-Wedge} we define pull-back and push-forward by
\begin{equation}\label{def_pullback}
\vp :=\Theta^*_p \phi := e^{-\beta_p x} \Psi \phi \quad \text{and} \quad \phi:= \Theta^p_* \vp:= \Psi^{-1}e^{\beta_p x} \vp
\end{equation}
with $\beta_p \in \R$. 
Let $\phi$ be the solution of \eqref{probl_laplace_neumann}, then by \cite[Chapter 4]{Grisvard:Elliptic-Problems} we have that
\begin{equation} \label{trans_laplace}
\Theta^*_p(\Delta \phi)= e^{-2x} (r_{\beta_p}(\partial_x) + \partial_\theta^2) \varphi,
\end{equation}
where
\begin{equation} \label{polynomial_operator}
r_{\beta_p} (\partial_x)\varphi := (\partial_x + \beta_p)^2 \vp.
\end{equation}
To absorb the factor $e^{-2 x}$ in \eqref{trans_laplace}, we set
\begin{equation} \label{transf_pullback_neumann}
g= \widetilde{\Theta}^*_p f := e^{2x} \Theta^*_p f
\end{equation}
with inverse $\widetilde{\Theta}^p_* := (\widetilde{\Theta}^*_p)^{-1}$.
By the choice 
\begin{equation} \label{def_beta}
\beta_p= 3 - \frac{2 +\gamma}{p}
\end{equation}
 Lemma \ref{transfomationen} implies that
\[
\widetilde{\Theta}^*_p \in \sL_{is}\left( \widehat{W}_\gamma^{1,p}(G), W^{1,p}(\Omega) \right).
\]
We notice that $\beta_p$, the pull-back and the push-forward depend on $p$. That means that the corresponding operator families may not be consistent in $p$.

After transforming the boundary conditions of \eqref{probl_laplace_neumann} to the layer domain we obtain
\begin{equation*}
\partial_\theta \vp = 0 \quad \text{on} \ \partial \Omega= \R \times \{0, \theta_0 \}.
\end{equation*}
Hence, \eqref{probl_laplace_neumann} is equivalent to

\begin{equation} \label{probl_transf_laplace}
\begin{array}{r@{\ =\ }lll}
-(r_{\beta_p}(\partial_x) + \partial_\theta^2) \vp  & g &  \text{in} & \Omega \\[0.5em]
\partial_\theta \vp   & 0 & \text{on} & \partial \Omega.  \\
\end{array}
\end{equation}

The proof of the well-posedness of problem \eqref{probl_transf_laplace} needs some preparation. We start to describe the operators associated to the single parts of \eqref{probl_transf_laplace}:

\begin{enumerate}
\item Let $r_{\beta_p}$ be the polynomial given in \eqref{polynomial_operator} with $\beta_p$ given as in \eqref{def_beta}. We define $\cT_{x}$ in $L^p(\R)$ by setting
\[
	\cT_{x} \vp := - r_{\beta_p}(\partial_x) \vp, \quad \vp \in D(\cT_{x}) := W^{2,p}(\R).
\]
The spectrum of $\cT_{x}$ is given by the parabola $- r_{\beta_p}(i \R)$, which is symmetric w.r.t. the real axis, open to the right and has its intersection point with the real axis at $- \beta_p^2$. It is known that $\cT_{x} + d \in \cH^{\infty} (L^p(\R))$ for $d> \beta_p^2$ with $\phi^\infty_{\cT_{x} +d} < \frac{\pi}{2}$, see \cite{Nazarov:Heat-Equation, Maier-Saal:Stokes-Wedge}.
These properties are also true for the canonical extension of $\cT_x$ to $L^p(\R, L^p(I))$, that is for the operator
$$T_x \vp:= \cT_x \vp, \quad \vp \in D(T_x):=W^{2,p}(\R, L^p(I)).$$
See, for instance, \cite{Weis:Maximal-Regularity, Denk-Hieber-Pruess:Maximal-Regularity, Kunstmann-Weis:Parabolic-Equations} for operator-valued Fourier multiplier results. \\[-0.5\baselineskip]

\item We define $\cT_{\theta}$ in $L^p(I)$ by setting
$$\cT_{\theta}  \vp:= -\partial^2_\theta \vp, \quad  \ \varphi \in D(\cT_{\theta}):=
\left\{\phi \in W^{2,p}(I): \ \partial_\theta \phi =0 \ \text{on} \ \partial I \right\}. $$
It is straight forward to calculate its spectrum, which is given as 
\begin{equation}\label{eigenwert_neumann}
\sigma(\cT_{\theta})= \left\{ 0 \right\} \cup \left\{ \left(\frac{\pi k}{\theta_0} \right)^2, \ k \in \N \right\} 
\end{equation}
with corresponding eigenfunctions
$$\tilde{e}_k(\theta):= \cos \left(\frac{\pi k}{\theta_0} \theta \right), \quad k\in \N_0, \ \theta \in I,  $$
see also \cite{Maier-Saal:Stokes-Wedge}. Since $\cT_{\theta}$ is self-adjoint in
$L^2(I)$, the eigenfunctions form a basis of $L^2(I)$. We denote by
$(\lambda_i)_{i \in \N_0}$ the set of eigenvalues of $\cT_{\theta}$, i.e.,
$(\lambda_i)_{i \in \N_0}= \sigma(\cT_{ \theta})$ such that
$\lambda_0<\lambda_1<\dots.$
Setting $e_0 := \tilde{e}_0 / \sqrt{\theta_0}$, where $\tilde{e}_0$ is the eigenfunction to the eigenvalue $\lambda_0=0$,
and $e_i := \sqrt{2} \tilde{e}_i / \sqrt{\theta_0}$, where $\tilde{e}_i$ is the eigenfunction to the eigenvalue $\lambda_i$ for all $i \in \N$, we have
$$ \left<e_i, e_j \right> = \frac{2}{\theta_0} \int_0^{\theta_0} \tilde{e}_i \cdot \tilde{e}_j \ d\theta = \delta_{ij}, \quad \quad i,j \in \N,$$
and 
\begin{align*}
\left< e_0,\ e_j \right> &= \frac{\sqrt{2}}{\theta_0} \int_{0}^{\theta_0} \tilde{e}_0 \cdot \tilde{e}_{j} \ d\theta =0, \quad \quad j \in \N,  \\[0.5em]
\left< e_0,\ e_0 \right> &= \frac{1}{\theta_0} \int_{0}^{\theta_0} \tilde{e}_0 \cdot \tilde{e}_{0} \ d\theta  =1.
\end{align*}
By Fourier series techniques it is straight forward to see that
$\cT_{\theta}$ admits an $\cH^\infty$-calculus on $L^q(I)$ with
$\phi^\infty_{\cT_{\theta}}=0$; see \cite{Duong:Functional-Calculus-Elliptic-Operators} for more details.
Again these facts remain valid for the canonical extension of
$\cT_{\theta}$ to $L^p(\R, L^p(I))$, which is defined by
$$T_{\theta} \vp:= \cT_{\theta} \vp, \quad D(T_{\theta})
:= L^p(\R, D(\cT_{\theta})). $$
\end{enumerate}

Optimal regularity for \eqref{probl_transf_laplace} is reduced to invertibility of the operator
\[ T_p:= T_x + T_\theta: D(T_p) \rightarrow L^p(\Omega)
\]
if we can show that
\begin{equation*}
D(T_p) = \left\{ \vp \in W^{2,p}(\Omega):  \ \partial_\theta \vp =0 \ \text{on} \ \partial \Omega \right\}
= D(T_x) \cap D( T_\theta).
\end{equation*}
To this end, for $m\in\N$ let	
\begin{equation*}
P^c_{m,p} \vp = \sum^m_{i=0}\left< \vp , e_i \right>e_i
\end{equation*}
be the projection of $\vp \in L^p(I)$ onto $\left< e_0, \dots , e_m \right>$ and put $P_{m,p}:=1-P^c_{m,p}$. 
We also set $E^p_m:=P_{m,p} \left( L^p(I) \right)$. It is obvious that
$\left( P_{m,p} \right)_{1<p<\infty}$ is a consistent family on
$(L^p(I))_{1<p<\infty}$, so we omit the index $p$  
and write $P_m$. If $\mP_m$ denotes the canonical extension of $P_m$ to 
$L^p(\R, L^p(I))$, then $\mP_m \in \sL (L^p(\Omega))$ is a projector onto $L^p(\R, E_m^p)$.
Consequently, we have the topological decomposition
\begin{equation*}
L^p(\Omega)=L^p(\R, \left< e_0,\dots, e_m \right>) \oplus L^p(\R, E_m^p).
\end{equation*}
The proof of the following properties is straight forward.

\begin{lemma}\label{properties}
Let $1< p < \infty$. Let $d >\beta_p^2$ with $\beta_p$ as given in \eqref{def_beta}, $m\in\N$ and $T_{x}$, $T_{\theta}$ be given as above. 
Then we have
\begin{enumerate}
\item $\mP_m \vp \in D(T_{x})$ and $\mP_m T_{x} \vp= T_{x} \mP_m \vp$
for $\vp \in D(T_{x})$,
\item $\mP_m \vp \in D(T_{\theta})$ and $\mP_m T_{\theta} \vp= T_{\theta} \mP_m \vp$ for $\vp \in D(T_{\theta})$,
 \item $T_{x} + d, T_{\theta } \in \cH^{\infty} \left( L^p(\R, E^p_m) \right) \ \cap \ \cH^{\infty} \left( L^p(\R, \left<e_0,\dots, e_m \right>) \right) $ with the corresponding angles $\phi^\infty_{T_{x} + d }< \frac{\pi}{2}$, $\phi^\infty_{T_{\theta}} =0$,
 \item $\mP_{m}$, $(\lambda-T_x)^{-1}$ and $(\mu- T_{\theta})^{-1}$
 commute pairwise for $\lambda \in \rho(T_{x})$ and 
 $\mu \in \rho(T_{\theta})$.
\end{enumerate}
\end{lemma}

The invertibility of $T_p=T_x+T_\theta$ essentially
follows by the operator sum
method. For instance one can apply \cite[Proposition 3.5]{Nau-Saal:Cyindrical-Boundary-Value-Problems}, which is a consequence of the Kalton-Weis theorem \cite[Cororallary 5.4]{Kalton-Weis:Operator-Sums}. 

\begin{proposition}\label{isomorph_op_neum}
Let $1<p<\infty$ and $\beta_p $ be defined as in \eqref{def_beta} with $\gamma := 0$.
Then
\begin{equation*}\label{isomtqa}
T_{p} \in \sL_{is}\bigl(D(T_p),\, L^p(\Omega)\bigr),
\end{equation*}
if and only if 
\begin{equation}\label{cond_eigenv}
{\beta_p}^2 \notin \sigma(T_{\theta})=\{(\pi k/\theta_0)^2,\,k\in\N_0\}.
\end{equation}
\end{proposition}
\begin{proof}
Relying on Lemma~\ref{properties}, the fact that
\[
T_{p} \in \sL_{is}\bigl(D(T_{x}) \cap D(T_{\theta}),\, L^p(\Omega)\bigr)
\]
follows by copying almost verbatim the lines of the proof of \cite[Theorem~2.3]{Koehne-Saal-Westermann:Stokes-Wedge}.
The proof of \cite[Lemma 2.5]{Koehne-Saal-Westermann:Stokes-Wedge} in addition shows that
\[
	W^{2,p}(\Omega)=W^{2,p}(\R,L^p(I))\cap L^p(\R,W^{2,p}(I)).
\]
The definition of the Sobolev space 
then yields
 that
$$D(T_{x})\cap D(T_\theta)=\left\{\varphi \in W^{2,p}(\Omega): \partial_\theta \vp=0 \text{ on } \partial \Omega \right\} = D(T_p). $$
This completes the proof.
\end{proof}

Next, we show higher regularity of the transformed problem \eqref{probl_transf_laplace}. 

\begin{corollary} \label{cor_higher_regularity}
Let $1<p< \infty$ and let $\beta_p $ be defined as in \eqref{def_beta} with $\gamma := 0$ and let condition \eqref{cond_eigenv} be fulfilled.
Then for every $g \in W^{1,p}(\Omega)$ the solution $\vp \in W^{2,p}(\Omega)$ of \eqref{probl_transf_laplace} satisfies the estimate
\begin{equation*}
\| \vp\|_{W^{3,p}(\Omega)} \leq C \| g \|_{W^{1,p}(\Omega)} 
\end{equation*}
for some constant $C>0$ that is independent of $\vp$ and $g$.
\end{corollary}

\begin{proof}
Denote by $D_1^h \vp$ the difference quotient
  $$ D_1^h \vp(x, \theta):= \frac{\vp((x, \theta) + h e_1) -\vp(x, \theta)}{h}, \quad \quad h \in \R, \ h \neq 0,$$
where $e_1:=(1,0) $. Let $ \vp \in D(T_p)$ be the solution of \eqref{probl_transf_laplace}.
Applying $D_1^h$ to \eqref{probl_transf_laplace} and using the fact that $D_1^h$ commutes with $T_{p}$, we obtain
\begin{align}\label{diff_gleichung}
 D_1^h T_p \vp &= D_1^h g \ \text{in} \ \mathcal{D}'(\Omega) \notag \\
 \Leftrightarrow \quad T_p D_1^h \vp &= D_1^h g \ \text{in} \ \mathcal{D}'(\Omega).
\end{align}

Now, let $g \in W^{1,p}(\Omega)$. The above calculation and Proposition \ref{isomorph_op_neum} imply
\begin{align}\label{absch_diffqu}
\|D_1^h \vp - D_1^{h'}\vp \|_{W^{2,p}(\Omega)} \leq  C  \|(D_1^h- D_1^{h'}) g \|_{L^p(\Omega)}  
\end{align}
for a constant $C>0$. 

For the right-hand side of \eqref{absch_diffqu} it is straight forward to see that
\begin{align*}
\|(D_1^h- D_1^{h'})g \|_{L^p(\Omega) }
 \underset{h, h' \rightarrow 0}{ \longrightarrow} 0,
\end{align*}
which implies that $D_1^h g$ is a Cauchy sequence in $L^p(\Omega)$ converging to $\partial_x g \in L^p(\Omega)$.

It follows by the estimate \eqref{absch_diffqu} that $D_1^h \vp $ is a Cauchy sequence in $ D(T_p) $ converging to $\partial_x \vp \in D(T_p)$.
The last calculations imply
\begin{align*}
\|\partial_x \vp\|_{W^{2,p}(\Omega)} & \leq C \|\partial_x g\|_{L^p(\Omega)} \leq C \| g \|_{W^{1,p}(\Omega)},
 \quad  \quad g \in W^{1,p}(\Omega),
\end{align*}
for a constant $C>0$. This yields that $\partial_x \vp \in W^{2,p}(\Omega)$, i.e.
\begin{equation*}
\vp, \nabla \vp, \nabla^2 \vp, \nabla^2 \partial_x \vp \in L^p(\Omega).
\end{equation*}
We still have to prove that  $\partial^3_\theta \vp \in L^p(\Omega)$. This, however, follows by $T_p \vp=g$. Since 
$$T_p \vp =( \partial^2_x + 2 \beta_p \partial_x + \beta^2_p + \partial_\theta^2 ) \vp$$
we have that
$$\partial^2_\theta \vp = - \left(  \partial_x^2 + 2 \beta_p  \partial_x + \beta_p^2 \right) \vp +  g \in W^{1,p}(\Omega),$$
and, hence, $\partial_\theta^3 \vp \in L^p(\Omega)$.
\end{proof}

Now, we consider the equivalence of problems \eqref{probl_laplace_neumann} and \eqref{probl_transf_laplace}. We define the Laplacian $A_L$ on the wedge domain as
\begin{equation}\label{op_laplace_neumann}
A_L \phi:= \Delta \phi, \qquad \phi \in D(A_L) := \left\{ \eta \in K^3_{p, \gamma}(G): \ \partial_\nu \eta= 0 \ \text{on} \ \Gamma \right\} . 
\end{equation}

\begin{lemma}\label{transfomationen}
Let $p \in (1, \infty)$, let $\gamma \in \R$ such that $\gamma \neq p-2$ and let $ \beta_p= 3- \frac{2 + \gamma}{p}$. Let $ \Theta_*^p, \ \Theta^*_p$ be defined as in \eqref{def_pullback} and let $ \widetilde{\Theta}^p_*, \widetilde{\Theta}^*_p $ be defined as in \eqref{transf_pullback_neumann}. Then we have
\begin{align*}
\widetilde{\Theta}^*_p \in \sL_{is} \left( \widehat{W}_{\gamma}^{1,p} (G), W^{1,p}(\Omega) \right), \quad \Theta^*_p  \in  \sL_{is} \left( D(A_L), D(T_p) \right)
\end{align*}
where $ \|\cdot \|_{D(A_L)}=\| \cdot \|_{K^3_{p, \gamma}(G)}  $ and $\|\cdot\|_{D(T_p)}=  \| \cdot \|_{W^{3,p}(\Omega)} .  $ \\
In particular, $\phi \in D(A_L)$ is the unique solution to \eqref{probl_laplace_neumann} for $f \in \widehat{W}_{\gamma}^{1,p}(G)$,
if and only if $\varphi= \Theta^*_p \phi \in D(T_p)$ is the unique solution to \eqref{probl_transf_laplace} for $g= \widetilde{\Theta}^*_p f \in W^{1,p}(\Omega)$. 
\end{lemma}

\begin{proof}
The fact that $  \Theta^*_p  \in  \sL_{is} \left( K^3_{p,\gamma}(G), W^{3,p}(\Omega)\right)$ follows from \cite[Manuscript~2, Lemma~B.3~(1)]{Westermann:Stokes-Equations} with $l-k:=3$.
In combination with the boundary conditions transformed at the beginning of this section, we obtain
\[
	\Theta^*_p \in \sL\left(D(A_L), D(T_p) \right) \quad \text{and} \quad \Theta_*^p \in \sL\left(D(T_p), D(A_L) \right).
\] 
Since $\Theta^*_p$ is the inverse of $\Theta^p_*$, the second assertion is follows.

Now, let $g \in W^{1,p}(\Omega)$.
Then we have
\begin{equation*}
	\begin{array}{rclclcl}
		\Psi(\partial_1 \widetilde{\Theta}^*_p g)(x, \theta)
			& = & c_p \cos \theta e^{\delta x / p} g(x, \theta) & + & \cos \theta e^{\delta x / p} \partial_x g(x, \theta) & - & \sin \theta e^{\delta x / p} \partial_\theta g(x, \theta), \\[0.5em]
		\Psi(\partial_2 \widetilde{\Theta}^*_p g)(x, \theta)
			& = & c_p \sin \theta e^{\delta x / p} g(x, \theta) & + & \sin \theta e^{\delta x / p} \partial_x g(x, \theta) & + & \cos \theta e^{\delta x / p} \partial_\theta g(x, \theta)
	\end{array}
\end{equation*}
with $\delta := (\beta_p - 3) p = 2 + \gamma$ and $c_p := \beta_p - 2$.
For $\varphi \in L_p(\Omega)$ and $\chi \in \{\,\cos,\,\sin\,\}$ and $\phi$ given as $\Psi \phi(x, \theta) = \chi(\theta) e^{- \delta x / p} \varphi(x, \theta)$
we further have
\begin{equation*}
	\|\phi\|^p_{L^p_\gamma(G)}
		= \int_G |\phi|^p \rho^\gamma d(x_1, x_2)
		= \int_\Omega e^{(2 + \gamma) x} |\chi(\theta) e^{- \delta x / p} \varphi(x, \theta)|^p d(x, \theta)
		\leq \|\varphi\|^p_{L_p(\Omega)},
\end{equation*}
which yields
\begin{equation*}
	\|\widetilde{\Theta}^p_* g\|_{\widehat{W}^{1,p}_\gamma(G)}
		\leq C \|g\|_{W^{1,p}(\Omega)}
\end{equation*}
for some constant $C>0$;
cf.~ also \cite[Manuscript~2, Lemma~B.1~(1)]{Westermann:Stokes-Equations} with $l:=2$.

Next, we show the converse estimate. Let $f \in \widehat{W}_\gamma^{1,p}(G)$ such that $f(0)=0$ if $\gamma < p-2$ and $f(\infty)=0$ if $\gamma> p-2$.
Then Hardy's inequality, see Lemma \ref{lemma_hardy}, implies
\begin{align*}
\|\widetilde{\Theta}^*_p f \|^p_{L^p(\Omega)}& =  \int_\Omega |e^{2x} e^{-\beta_p x} \Psi f(x, \theta) |^p d(x, \theta)\\[0.5em]
& = \int_G | \rho^{2 -(3- \frac{2+ \gamma}{p})} f(x_1,x_2) |^p \rho^{-2} d(x_1, x_2) \\[0.5em]
& = \| \rho^{-1} f \|^p_{L_\gamma^p(G)} \leq C \| \nabla f \|^p_{L_\gamma^p(G)}
\end{align*} 
for some constant $C:= C(p,\gamma)>0$. Moreover, we have 
\begin{align*}
&\|  \widetilde{\Theta}^*_p f \|^p_{\widehat{W}^{1,p}(\Omega)}= \int_\Omega | \nabla e^{(2-\beta_p) x} \Psi  f(x,\theta) |^p d(x,\theta) \\[0.5em]
& \quad = \int_G \left| \rho^{2 -\beta_p} \left( \begin{pmatrix} 2- \beta_p \\ 0 \end{pmatrix} f(x_1, x_2) +\rho \left(\begin{array}{rr} \cos\theta & \ \sin \theta \\ - \sin\theta & \ \cos\theta \end{array} \right)\right)  \nabla f(x_1, x_2)\right|^p \!\! \rho^{-2}d(x_1, x_2) \\[0.5em]
& \quad \leq \int_G | \rho^{2 - (3 - \frac{2+ \gamma}{p})} \begin{pmatrix} 2 - \beta_p  \\ 0 \end{pmatrix} f(x_1, x_2) |^p \rho^{-2} d(x_1, x_2) \\[0.5em]
& \quad \quad \quad \quad + \int_G | \rho^{3- (3- \frac{2 +\gamma}{p})}\nabla f(x_1,x_2) |^p \rho^{-2} d(x_1, x_2) \\[0.5em]
& \quad \leq C \big(\| \rho^{-1} f \|^p_{L_\gamma^p(G)} + \| \nabla f \|^p_{L_\gamma^p(G)}\big) \leq C \| \nabla f \|^p_{L_\gamma^p(G)},
\end{align*}
for a constant $C:=C(p, \gamma)>0$.
Hence, the first assertion $ \widetilde{\Theta}^*_p  \in \sL_{is} ( \widehat{W}_\gamma^{1,p} (G), W^{1,p}(\Omega) )$ follows.
\end{proof}

\begin{remark}\label{remark_spectral_condition}
(a) For $ \beta_p = 3 - \frac{2+ \gamma}{p}$ the condition \eqref{cond_eigenv} is fulfilled, if every eigenvalue $\lambda_i$ of $T_{p, \theta}$ satisfies
\begin{equation}
\lambda_i \neq \beta_p ^2= \left(3- \frac{2+ \gamma}{p}\right)^2.
\end{equation}
For the case $\gamma =0$, i.e.\ for the Kondrat'ev weight $\rho^\gamma \equiv 1$, we then have
$$  \lambda_i \neq \beta_p^2 \ \Leftrightarrow \ \left(3 - \frac{2}{p}\right)^2 \neq \left( \frac{i \pi}{\theta_0} \right)^2, \qquad i \in \N_0. $$
This is equivalent to
$$ p \neq \frac{2 \theta_0}{3 \theta_0 - i \pi}, \qquad i \in \N_0.$$
Since $\theta_0 \in (0, \pi)$, the above relation is always fulfilled for $p \in (1, \infty) \setminus \{ \frac{2 \theta_0}{3 \theta_0- \pi},\  \frac{2 \theta_0}{3 \theta_0-2 \pi} \} $.

(b) Lemma \ref{transfomationen} is fulfilled for all $p \in (1, \infty)$ such that $\gamma \neq p-2$ with $\gamma \in \R$. For $\gamma = 0$ this is equivalent to $ p\neq 2$.
\end{remark}
Proposition \ref{isomorph_op_neum}, Corollary \ref{cor_higher_regularity},  Lemma \ref{transfomationen} and the last remark yield the following result.
\begin{corollary}\label{cor_probl_summary}
Let $p \in (1, \infty) \setminus \{ \frac{2 \theta_0}{3 \theta_0- \pi},\ \frac{2 \theta_0}{3 \theta_0 - 2 \pi} ,\ 2 \} $, $\theta_0 \in (0, \pi)$ and $ \rho= |(x_1,x_2)|$. Then equation \eqref{probl_laplace_neumann} is for each $f \in \widehat{W}^{1,p}(G)$ uniquely solvable with a solution $\phi$ satisfying 
$$ \rho^{|\alpha|-3} \partial^{\alpha} \phi \in L^p(G), \qquad |\alpha| \leq 3.$$
\end{corollary}

The next corollary generalizes the above result to time dependent data.

\begin{corollary} \label{cor_probl_zeit}
Let $p \in  (1, \infty) \setminus \{ \frac{2 \theta_0}{3 \theta_0- \pi},\ \frac{2 \theta_0}{3 \theta_0 - 2 \pi} ,\ 2\,  \} $, $\theta_0 \in (0, \pi)$ and $\rho= |(x_1, x_2)| $. Let $J=(0,T)$ with $0 < T < \infty$.
Then for every $f \in L^p(J, \widehat{W}^{1,p}( G))$ the equation
\begin{equation} \label{probl_laplace_neumann_zeit}
	 	\begin{array}{r@{\ =\ }lll}
	 		 \Delta \phi &  f  &\text{in} & J \times G, \\[0.5em]
	 		  \partial_\nu \phi  & 0& \text{on} & J \times \Gamma
	 	\end{array}
\end{equation}
has a unique solution $\phi$ satisfying
$$ \rho^{|\alpha|-3} \partial^{\alpha} \phi \in L^p(J, L^p(G)), \qquad |\alpha| \leq 3.$$
\end{corollary}

\begin{proof}
Assume that $f \in C^\infty(\overline{J \times G}) \cap L^p(J, \widehat{W}^{1,p}(G))$.
For every $t > 0$ let $\phi(t,\,\cdot\,) \in K^3_p(G)$ be the unique solution to the problem
$$\Delta \phi(t, \cdot) = f(t, \cdot) \quad \text{in} \ G, \qquad  \partial_\nu \phi(t, \cdot) =0 \quad \text{on} \ \Gamma,  $$
which exists due to the Corollary~\ref{cor_probl_summary}.
Now, we have
\begin{align*}
\| \phi \|^p_{L^p(J, K^3_p(G))} & = \int_0^T \| \phi(t,  \cdot) \|^p_{K^3_p(G)}  dt  \leq C^p \int_0^T \| f(t, \cdot) \|^p_{\widehat{W}^{1,p}(G)} dt = C^p \| f \|^p_{L^p(J, \widehat{W}^{1,p}( G))}
\end{align*}
for a constant $C>0$ which is independent of $u, \ f$  and $t > 0$.
This shows unique solvability of \eqref{probl_laplace_neumann_zeit} for a right-hand side $f \in C^\infty(\overline{J \times G}) \cap L^p(J, \widehat{W}^{1,p}(G))$.
Now, since the latter space is dense in $L^p(J, \widehat{W}^{1,p}(G))$, an approximation argument yields the assertion for every right-hand side $f \in L^p(J, \widehat{W}^{1,p}(G))$.
\end{proof}

\section{The Stokes Equations subject to Navier Boundary Conditions}\label{sec_navierslipbc}
Let $J=(0,T)$ with $0 < T < \infty$ and let $G \subset \R^2 $ be the wedge defined as in \eqref{def_wedge} with opening angle $\theta_0 \in (0, \pi)$.
The aim of this section is to prove Theorem~\ref{thm_navierslip}, that is the unique solvability of problem \eqref{problem_navierslipbc} in the $L^p$-setting for all \linebreak $p \in (1, \infty) \setminus \{\frac{2 \theta_0}{3 \theta_0 -  \pi},\ \frac{2 \theta_0}{3 \theta_0 - 2 \pi},\ \frac{3}{2},\ 2,\ 3 \}$.
We start with a proof of the well-posedness of the Stokes equations subject to inhomogeneous perfect slip boundary conditions. 

\subsection{Inhomogeneous Perfect Slip Boundary Conditions}
Let $\mE$ and $\mF$ be defined as in \eqref{solution_class} and \eqref{data_class}, respectively.
Here we consider the system
\begin{equation}\label{probl_perfect_slip}
	 	\begin{array}{r@{\ =\ }lll}
	 		 \partial_t u - \Delta u + \nabla p &  f &
			\text{in} & J \times G, \\[0.25em]
			\div u & g& \text{in } & J \times G ,\\[0.25em]
		    \text{curl} \ u &h_1 & \text{on} & J \times \Gamma , \\[0.25em]
	 		u \cdot \nu &h_0 & \text{on} & J \times \Gamma ,\\[0.25em]
			u(0) & u_0& \text{in} & G,
	 	\end{array}
\end{equation}
where the boundary of $G$ is decomposed as in \eqref{def_boundary} as $\partial G = \Gamma \cup \{\,0\,\}$
with its smooth part given as $\Gamma = \Gamma_1 \cup \Gamma_2$.
Recall that $(\tau,\,\nu) = (\tau_j,\,\nu_j)$ for $j = 1,\,2$ denotes the positively oriented pair of unit tangential and unit outer normal vector on $\Gamma_j$ as introduced in Section~\ref{sec_intro}.
Of course, the boundary conditions in \eqref{probl_perfect_slip} have to be understood as
\begin{align*}
	\text{curl}\ u &= h^{(1)}_1 \quad \text{on} \ J \times \Gamma_1, \\
	\text{curl}\ u &= h^{(2)}_1 \quad \text{on} \ J \times \Gamma_2, \\
	 u \cdot \nu_1 &= h^{(1)}_0 \quad \text{on} \ J \times \Gamma_1, \\
	 u \cdot \nu_2 &= h^{(2)}_0 \quad \text{on} \ J \times \Gamma_2,
\end{align*}
where $h^{(j)}_\ell = h_\ell|_{\Gamma_j}$ for $\ell = 0,\,1$ and $j = 1,\,2$.
We aim at solutions $(u,p) \in \mE$
and, hence, the given data in \eqref{probl_perfect_slip} have to satisfy the regularity conditions
$$(f,\ g, \ h_1, \ h_0,\ u_0) \in \mF. $$
In order to treat problem \eqref{probl_perfect_slip} we first need the following result concerning traces on the wedge domain $G$.

\begin{proposition}\label{prop_trace}
Let $J=(0,T)$ with $0 < T < \infty$ and let $G \subseteq \R^2$ be the wedge domain defined as in \eqref{def_wedge} with opening angle $\theta_0 \in (0, \pi)$.
Let $1<p<\infty$ with $p \neq 2$.
Furthermore, let $\Gamma_1 = (-\infty,\,0) \cdot \tau_1$ and $\Gamma_2 = (0,\,\infty) \cdot \tau_2$ with
\begin{equation*}
	\tau_1 = -e_1, \quad \nu_1 = -e_2, \quad \tau_2 = (\cos \theta_0,\,\sin \theta_0)^T, \quad \nu_2 = (- \sin \theta_0,\,\cos \theta_0)^T
\end{equation*}
such that $\partial G = \Gamma_1 \overset{.}{\cup} \Gamma_2 \overset{.}{\cup} \{ 0 \} $.
Now, suppose that
\begin{equation*}
	\begin{array}{rcll}
		g_j & \in & W^{1-1/2p}_p(J, L^p(\Gamma_j)) \cap L^p(J, W^{2- 1/p}_p(\Gamma_j)),     & \qquad j =1,\,2, \\[0.5em]
		h_j & \in & W^{1/2 - 1/2p}_p(J, L^p(\Gamma_j)) \cap L^p(J, W^{1- 1/p}_p(\Gamma_j)), & \qquad j =1,\,2,
	\end{array}
\end{equation*}
such that 
\begin{equation*}
	\begin{array}{rcll}
		  \langle g_1 \rangle_1 & = & \langle g_2 \rangle_2 & \quad \textrm{in}\ J, \\[0.25em]
		  \langle \partial_{\tau_1} g_1 \rangle_1 + \cos \theta_0 \cdot \langle \partial_{\tau_2} g_2 \rangle_2 & = & \sin \theta_0 \cdot \langle h_2 \rangle_2 & \quad \textrm{in}\ J, \quad \textrm{if}\ p > 2, \\[0.25em]
		- \langle \partial_{\tau_2} g_2 \rangle_2 - \cos \theta_0 \cdot \langle \partial_{\tau_1} g_1 \rangle_1 & = & \sin \theta_0 \cdot \langle h_1 \rangle_1 & \quad \textrm{in}\ J, \quad \textrm{if}\ p > 2.
	\end{array}
\end{equation*}
Then there exists a function $u \in W^{1,p}(J, L^p(G)) \cap L^p(J, W^{2,p}(G))$ that satisfies
\begin{equation*}
	\begin{array}{rclcrcll}
		u & = & g_1 & \quad \textrm{and} \quad & \partial_{\nu_1} u & = & h_1 & \quad \textrm{on}\ J \times \Gamma_1, \\[0.25em]
		u & = & g_2 & \quad \textrm{and} \quad & \partial_{\nu_2} u & = & h_2 & \quad \textrm{on}\ J \times \Gamma_2.
	\end{array}
\end{equation*}
\end{proposition}

\begin{proof}
{\itshape Step 1.}
We first show that we can w.\,l.\,o.\,g.\ assume that $\langle g_j \rangle_j = 0$ as well as $\langle \partial_{\tau_j} g_j \rangle_j = \langle h_j \rangle_j = 0$, if $p > 2$, for $j = 1,\,2$.
Indeed, there exist extensions
\begin{equation*}
	\begin{array}{rcl}
		\hat{g}_1 & \in & W^{1 - 1/2p}_p(J, L^p(\Sigma_1)) \cap L^p(J, W^{2 - 1/p}_p(\Sigma_1)),  \\[0.5em]
		\hat{h}_1 & \in & W^{1/2 - 1/2p}_p(J, L^p(\Sigma_1)) \cap L^p(J, W^{1 - 1/p}_p(\Sigma_1))
	\end{array}
\end{equation*}
of $g_1$ and $h_1$, respectively, to the hyperplane $\Sigma_1 := \R \cdot \tau_1$; cf.~\cite[Thm.\ 4.26]{Adams-Fournier:Sobolev-Spaces}.
Now, the trace theory for anisotropic function spaces on the halfspace implies that there exists
\begin{equation*}
	v \in W^{1, p}(J, L^p(\R \times (0, \infty))) \cap L^p(J, W^{2, p}(\R \times (0, \infty)))
\end{equation*}
such that $v = \hat{g}_1$ and $\partial_{\nu_1} v = \hat{h}_1$ on $J \times \Sigma_1$.
Then we set $u = v + \hat{u}$ and infer that $\hat{u} \in W^{1,p}(J, L^p(G)) \cap L^p(J, W^{2,p}(G))$ has to satisfy the boundary conditions
\begin{equation*}
	\begin{array}{rclcrcll}
		\hat{u} & = & 0         & \quad \textrm{and} \quad & \partial_{\nu_1} \hat{u} & = & 0         & \quad \textrm{on}\ J \times \Gamma_1, \\[0.25em]
		\hat{u} & = & \hat{g}_2 & \quad \textrm{and} \quad & \partial_{\nu_2} \hat{u} & = & \hat{h}_2 & \quad \textrm{on}\ J \times \Gamma_2
	\end{array}
\end{equation*}
for $\hat{g}_2 := g_2 - v|_{\Gamma_2}$ and $\hat{h}_2 := h_2 - \partial_{\nu_2} v$.
Due to the choice of $v$ and the compatibility conditions for the boundary data we have $\langle \hat{g}_2 \rangle_2 = \langle g_2 \rangle_2 - \langle g_1 \rangle_1 = 0$ and
\begin{equation*}
	\begin{array}{rcl}
		\langle \partial_{\tau_2} \hat{g}_2 \rangle_2
			& = & \langle \partial_{\tau_2} g_2 \rangle_2 - \langle \partial_{\tau_2} v \rangle_2 \\[0.5em]
			& = & - \cos \theta_0 \cdot \langle \partial_{\tau_1} g_1 \rangle_1 - \sin \theta_0 \cdot \langle h_1 \rangle_1 - \langle \partial_{\tau_2} v \rangle_2 \\[0.5em]
			& = & \cos \theta_0 \cdot \langle \partial_{x_1} v \rangle_\bullet + \sin \theta_0 \cdot \langle \partial_{x_2} v \rangle_\bullet - \langle \partial_{\tau_2} v \rangle_2 = 0, \qquad \textrm{if}\ p > 2,
	\end{array}
\end{equation*}
as well as
\begin{equation*}
	\begin{array}{rcl}
		\langle \hat{h}_2 \rangle_2
			& = & \langle h_2 \rangle_2 - \langle \partial_{\nu_2} v \rangle_2 \\[0.5em]
			& = & \frac{1}{\sin \theta_0} \big( \langle \partial_{\tau_1} g_1 \rangle_1 + \cos \theta_0 \cdot \langle \partial_{\tau_2} g_2 \rangle_2 \big) - \langle \partial_{\nu_2} v \rangle_\bullet \\[0.5em]
			& = & \frac{1}{\sin \theta_0} \big( \langle \partial_{\tau_1} g_1 \rangle_1 - \cos^2 \theta_0 \cdot \langle \partial_{\tau_1} g_1 \rangle_1 - \sin \theta_0 \cdot \cos \theta_2 \cdot \langle h_1 \rangle_1 \big) - \langle \partial_{\nu_2} v \rangle_2 \\[0.5em]
			& = & \sin \theta_0 \cdot \langle \partial_{\tau_1} g_1 \rangle_1 - \cos \theta_0 \cdot \langle h_1 \rangle_1 - \langle \partial_{\nu_2} v \rangle_2 \\[0.5em]
			& = & - \sin \theta_0 \cdot \langle \partial_{x_1} v \rangle_\bullet + \cos \theta_0 \cdot \langle \partial_{x_2} v \rangle_\bullet - \langle \partial_{\nu_2} v \rangle_2 = 0, \qquad \textrm{if}\ p > 2.
	\end{array}
\end{equation*}
Hence, $\langle \partial_{\tau_2} \hat{g}_2 \rangle_2 = \langle \hat{h}_2 \rangle_2 = 0$, if $p > 2$.

{\itshape Step 2.}
Now, assume that $\langle g_j \rangle_j = 0$ as well as $\langle \partial_{\tau_j} g_j \rangle_j = \langle h_j \rangle_j = 0$, if $p > 2$, for $j = 1,\,2$.
Let $\widetilde{G} :=  (0, \infty)^2$ be the wedge domain with opening angle $\frac{\pi}{2}$.
Here we set $\widetilde{\Gamma}_1 := \Gamma_1$ and $\widetilde{\Gamma}_2 := \{\,0\,\} \times (0, \infty)$ to obtain the decomposition $\partial \widetilde{G} = \widetilde{\Gamma}_1 \overset{.}{\cup} \widetilde{\Gamma}_2 \overset{.}{\cup} \{\,0\,\}$ of the boundary of $\widetilde{G}$.
We abbreviate $\rho := |x| = |(x_1,x_2)|$ for $x \in \R^2$ and define a transformation
$$\Phi: G \rightarrow \widetilde{G}, \qquad \Phi(x_1, x_2) = \left( \rho \cos \left( \frac{\pi}{2\theta_0} \arccos \left(\frac{x_1}{\rho}\right) \right), \ \rho \sin \left( \frac{\pi}{2\theta_0} \arccos \left(\frac{x_1}{\rho} \right) \right) \right).$$
It is not difficult to see that $\Phi: G \rightarrow \widetilde{G}$ is a $C^\infty$-diffeomorphism.
We set $\widetilde{g}_1 := g_1$, $\widetilde{h}_1 := h_1$ as well as
\begin{equation*}
	\widetilde{g}_2(t,s e_2) := g_2(t,  s \tau_2), \qquad \widetilde{h}_2(t,s e_2) := h_2(t,  s \tau_2), \qquad t \in J,\ s > 0.
\end{equation*}
Then we have 
\begin{equation*}
	\begin{array}{rcll}
		\widetilde{g}_j & \in & W^{1-1/2p}_p(J, L^p(\widetilde{\Gamma}_j)) \cap L^p(J, W^{2- 1/p}_p(\widetilde{\Gamma}_j)),     & \qquad j =1,\,2, \\[0.5em]
		\widetilde{h}_j & \in & W^{1/2 - 1/2p}_p(J, L^p(\widetilde{\Gamma}_j)) \cap L^p(J, W^{1- 1/p}_p(\widetilde{\Gamma}_j)), & \qquad j =1,\,2,
	\end{array}
\end{equation*}
and $\lim_{s \rightarrow 0} \widetilde{g}_j(t, s e_j) = 0$ as well as $\lim_{s \rightarrow 0} \partial_{x_j} \widetilde{g}_j(t, s e_j) = \lim_{s \rightarrow 0} \widetilde{h}_j(t, s e_j) = 0$, if $p > 2$, for $t \in J$ and $j = 1,\,2$.
Now, we apply \cite[Theorem VIII.1.8.5]{Amann:Parabolic-Problems-2}, which shows that there exists $\widetilde{u} \in W^{1,p}(J, L^p(\widetilde{G})) \cap L^p(J, W^{2,p}(\widetilde{G}))$ satisfying
\begin{equation*}
	\begin{array}{rclcrcll}
		\widetilde{u} & = & \widetilde{g}_1 & \quad \textrm{and} \quad & \partial_{\nu_1} \widetilde{u} & = & \widetilde{h}_1 & \quad \textrm{on}\ J \times \widetilde{\Gamma}_1, \\[0.25em]
		\widetilde{u} & = & \widetilde{g}_2 & \quad \textrm{and} \quad & \partial_{\nu_2} \widetilde{u} & = & \widetilde{h}_2 & \quad \textrm{on}\ J \times \widetilde{\Gamma}_2.
	\end{array}
\end{equation*}
Finally, we set $u = \widetilde{u} \circ \Phi \in W^{1,p}(J, L^p(G)) \cap L^p(J, W^{2,p}(G))$.
By construction, $u$ satisfies all desired boundary conditions.
Note that we indeed have $u \in L^p(J, W^{2,p}(G))$, which can be seen as follows:
We have $\partial_j \Phi \sim \rho^0$ as $\rho \rightarrow 0$ or $\rho \rightarrow \infty$ for $j = 1,\,2$ for the first derivatives of $\Phi$
and $\partial_j \partial_k \Phi \sim \rho^{-1}$ as $\rho \rightarrow 0$ and $\rho \rightarrow \infty$ for $j,\,k = 1,\,2$ for the second derivatives of $\Phi$,
i.\,e.\ $\partial_j \Phi_n,\ \rho \partial_j \partial_k \Phi_n \in L^\infty(G)$ for $j,\,k,\,n = 1,\,2$.
Moreover, $\det \nabla \Phi \equiv \frac{\pi}{2 \theta_0}$.
However, the chain rule shows that
\begin{equation*}
	\partial_j \partial_k (\widetilde{u} \circ \Phi)
		= \!\! \sum^2_{m, n = 1} \!\! \big( (\partial_m \partial_n \widetilde{u}) \circ \Phi \big) \partial_j \Phi_m \partial_k \Phi_n + \sum^2_{n = 1} \big( (\partial_n \widetilde{u}) \circ \Phi \big) \partial_j \partial_k \Phi_n,
			\quad j,\,k = 1,\,2
\end{equation*}
and we have $\rho^{-1} \partial_j \widetilde{u} \in L^p(J, L^p(\widetilde{G}))$ for $j = 1,\,2$ due to Hardy's inequality; cf.~Lemma~\ref{lemma_hardy}.
Note that by construction we have $\partial_j \widetilde{u}(\,\cdot\,,0) = 0$ in $J$ for $j = 1,\,2$, if $p > 2$,
since $\lim_{s \rightarrow 0} \widetilde{h}_j(t, s e_j) = 0$, for $t \in J$ and $j = 1,\,2$, if $p > 2$.
\end{proof}

\begin{remark}\label{rem_trace}
(a) For $\theta_0 = \frac{\pi}{2}$ we have $\cos \theta_0 = 0$ and $\sin \theta_0 = 1$
as well as $\tau_1 = - e_1$, \linebreak $\nu_1 = - e_2$, $\tau_2 = e_2$ and $\nu_2 = -e_1$.
In this case the compatibility conditions in Proposition~\ref{prop_trace} read
\begin{equation*}
	\begin{array}{rcll}
		  \langle g_1 \rangle_1 & = & \langle g_2 \rangle_2 & \qquad \textrm{in}\ J, \\[0.25em]
		- \langle \partial_{x_1} g_1 \rangle_1 & = & \langle h_2 \rangle_2 & \qquad \textrm{in}\ J, \quad \textrm{if}\ p > 2, \\[0.25em]
		- \langle \partial_{x_2} g_2 \rangle_2 & = & \langle h_1 \rangle_1 & \qquad \textrm{in}\ J, \quad \textrm{if}\ p > 2.
	\end{array}
\end{equation*}
These are precisely the compatibility conditions \cite[(VIII.1.8.7) \& (VIII.1.8.8)]{Amann:Parabolic-Problems-2}.
This is not surprising, since for $\theta_0 = \frac{\pi}{2}$ Proposition~\ref{prop_trace} is a special case of \cite[Thm.\ VIII.1.8.5]{Amann:Parabolic-Problems-2}.

(b) The value $p = 2$ is excluded in Proposition~\ref{prop_trace} due to technical reasons, only.
Indeed, for the space $W^2_p(G)$ instead of $W^{1,p}(J, L^p(G)) \cap L^p(J, W^{2,p}(G))$ a similar result is already available also for $p = 2$.
For instance, \cite[Theorem~1.5.2.8]{Grisvard:Elliptic-Problems}, which is formulated for a bounded domain with curvilinear boundary,
and its proof can obviously be transferred to the wedge domain $G$.
However, the compatibility conditions that arise for $p \neq 2$ in form of {\itshape continuity conditions} as in Proposition~\ref{prop_trace}
and \cite[Theorem~1.5.2.8]{Grisvard:Elliptic-Problems}~(a) have to be replaced by {\itshape integrability conditions} as in \cite[Theorem~1.5.2.8]{Grisvard:Elliptic-Problems}~(b) for $p = 2$.
The phenomenon that the case $p = 2$ requires different compatibility conditions is also related to the fact that Hardy's inequality,
which is used in the proof of Proposition~\ref{prop_trace} in form of Lemma ~\ref{lemma_hardy}, is not available for $p = 2$.
\end{remark}

\begin{corollary}\label{cor_trace}
Let $J=(0,T)$ with $0 < T < \infty$ and let $G \subset \R^2$ be the wedge domain defined as in \eqref{def_wedge} with opening angle $\theta_0 \in (0, \pi)$.
Let $1<p<\infty$ with $p \neq 2$.
Furthermore, let $\Gamma_1 = (-\infty,\,0) \cdot \tau_1$ and $\Gamma_2 = (0,\,\infty) \cdot \tau_2$ with
\begin{equation*}
	\tau_1 = -e_1, \quad \nu_1 = -e_2, \quad \tau_2 = (\cos \theta_0,\,\sin \theta_0)^T, \quad \nu_2 = (- \sin \theta_0,\,\cos \theta_0)^T
\end{equation*}
such that $\partial G = \Gamma_1 \overset{.}{\cup} \Gamma_2 \overset{.}{\cup} \{ 0 \}$ and set $\Gamma = \Gamma_1 \cup \Gamma_2$.
Now, suppose that
\begin{equation*}
	\begin{array}{rcll}
		h^{(j)}_0 & \in & W^{1-1/2p}_p(J, L^p(\Gamma_j)) \cap L^p(J, W^{2- 1/p}_p(\Gamma_j)),     & \qquad j =1,\,2, \\[0.5em]
		h^{(j)}_1 & \in & W^{1/2 - 1/2p}_p(J, L^p(\Gamma_j)) \cap L^p(J, W^{1- 1/p}_p(\Gamma_j)), & \qquad j =1,\,2,
	\end{array}
\end{equation*}
such that $\jump{h_1} = 0$ in $J$, if $p > 2$, and
\begin{equation*}
	\langle \partial_{\tau_1} h_0 \rangle_1 + \langle \partial_{\tau_2} h_0 \rangle_2 = \trace{h_1} \quad \textrm{in}\ J,
		\qquad \textrm{if}\ \theta_0 = {\textstyle \frac{\pi}{2}}\ \textrm{and}\ p > 2.
\end{equation*}
Then there exists a function $u \in W^{1,p}(J, L^p(G, \R^2)) \cap L^p(J, W^{2,p}(G, \R^2))$ that satisfies
\begin{equation}\label{prob_curl_bc}
	u \cdot \nu = h_0 \quad \textrm{and} \quad \mbox{curl}\,u = h_1 \quad \textrm{on}\ J \times \Gamma.
\end{equation}
\end{corollary}

\begin{proof}
First note that for $v \in W^{1,p}(J, L^p(G, \R^2)) \cap L^p(J, W^{2,p}(G, \R^2))$ we have
\begin{equation*}
	\text{curl}\,v = \partial_1 v_2 - \partial_2 v_1 = \partial_\tau (v \cdot \nu) - \partial_\nu (v \cdot \tau) \quad \text{on} \ J \times \Gamma.
\end{equation*}
Hence, if $v \cdot \nu = h_0$ and $\mbox{curl}\,v = h_1$ on $J \times \Gamma$, then $\partial_\nu (v \cdot \tau) = \partial_\tau h_0 - h_1$ on $J \times \Gamma$.

Now, we choose $g_j \in W^{1-1/2p}_p(J, L^p(\Gamma_j)) \cap L^p(J, W^{2- 1/p}_p(\Gamma_j))$ for $j = 1,\,2$ such that
\begin{equation*}
	\begin{array}{rcrl}
		\cos \theta_0 \cdot \partial_{\tau_1} g_1 & = & - (1 - \cos \theta_0) \cdot \partial_{\tau_1} h^{(1)}_0 + \frac{1}{2} \sin^2 \theta_0 \cdot h^{(1)}_1 & \qquad \textrm{at}\ J \times \{ 0 \}, \\[0.5em]
		\cos \theta_0 \cdot \partial_{\tau_2} g_2 & = &   (1 - \cos \theta_0) \cdot \partial_{\tau_2} h^{(2)}_0 - \frac{1}{2} \sin^2 \theta_0 \cdot h^{(2)}_1 & \qquad \textrm{at}\ J \times \{ 0 \}
	\end{array}
\end{equation*}
and $g_j(\,\cdot\,,0) = 0$ in $J$ for $j = 1,\,2$, if $\theta_0 \neq \frac{\pi}{2}$ and $p > 2$, and $g_j := 0$ for $j = 1,\,2$, if $\theta_0 = \frac{\pi}{2}$.

Next, we define $\widetilde{h}^{(j)}_0 \in W^{1-1/2p}_p(J, L^p(\Gamma_j)) \cap L^p(J, W^{2- 1/p}_p(\Gamma_j))$ for $j = 1,\,2$ as
\begin{equation*}
	\begin{array}{rcll}
		\sin \theta_0 \cdot \widetilde{h}^{(1)}_0(t, - s \tau_1) & := &   h^{(2)}_0(t,   s \tau_2) + \cos \theta_0 \cdot h^{(1)}_0(t, - s \tau_1) + g_1(t, - s \tau_1), & \quad t \in J,\ s > 0, \\[0.5em]
		\sin \theta_0 \cdot \widetilde{h}^{(2)}_0(t,   s \tau_2) & := & - h^{(1)}_0(t, - s \tau_1) - \cos \theta_0 \cdot h^{(2)}_0(t,   s \tau_2) + g_2(t,   s \tau_2), & \quad t \in J,\ s > 0,
	\end{array}
\end{equation*}
and $H_0 \in W^{1-1/2p}_p(J, L^p(\Gamma, \R^2)) \cap L^p(J, W^{2- 1/p}_p(\Gamma, \R^2))$ as $H_0 := \widetilde{h}_0 \cdot \tau + h_0 \cdot \nu$.
By construction we then have $H_0 \cdot \nu = h_0$ on $J \times \Gamma$.

Finally, we define $\widetilde{h}^{(j)}_1 \in W^{1/2 - 1/2p}_p(J, L^p(\Gamma_j)) \cap L^p(J, W^{1 - 1/p}_p(\Gamma_j))$ for $j = 1,\,2$ as
\begin{equation*}
	\begin{array}{rcll}
		\sin \theta_0 \cdot \widetilde{h}^{(1)}_1(t, - s \tau_1) & := & (\partial_{\tau_2} g_2)(t,   s \tau_2) + (1 - \cos \theta_0) \cdot (\partial_{\tau_1} h^{(1)}_0)(t, - s \tau_1), & \qquad t \in J,\ s > 0, \\[0.5em]
		\sin \theta_0 \cdot \widetilde{h}^{(2)}_1(t,   s \tau_2) & := & (\partial_{\tau_1} g_1)(t, - s \tau_1) - (1 - \cos \theta_0) \cdot (\partial_{\tau_2} h^{(2)}_0)(t,   s \tau_2), & \qquad t \in J,\ s > 0,
	\end{array}
\end{equation*}
and $H_1 \in W^{1/2 - 1/2p}_p(J, L^p(\Gamma, \R^2)) \cap L^p(J, W^{1 - 1/p}_p(\Gamma, \R^2))$ as $H_1 := (\partial_\tau h_0 - h_1) \cdot \tau + \widetilde{h}_1 \cdot \nu$.
By construction we then have $H_1 \cdot \tau = \partial_\tau h_0 - h_1$ on $J \times \Gamma$.

Now, it is readily checked that
\begin{equation*}
	\begin{array}{rcll}
		  \langle H_0 \rangle_1 & = & \langle H_0 \rangle_2 & \quad \textrm{in}\ J, \\[0.25em]
		  \langle \partial_{\tau_1} H_0 \rangle_1 + \cos \theta_0 \cdot \langle \partial_{\tau_2} H_0 \rangle_2 & = & \sin \theta_0 \cdot \langle H_1 \rangle_2 & \quad \textrm{in}\ J, \quad \textrm{if}\ p > 2, \\[0.25em]
		- \langle \partial_{\tau_2} H_0 \rangle_2 - \cos \theta_0 \cdot \langle \partial_{\tau_1} H_0 \rangle_1 & = & \sin \theta_0 \cdot \langle H_1 \rangle_1 & \quad \textrm{in}\ J, \quad \textrm{if}\ p > 2,
	\end{array}
\end{equation*}
Therefore, due to Proposition~\ref{prop_trace} there exists $u \in W^{1,p}(J, L^p(G, \R^2)) \cap L^p(J, W^{2,p}(G, \R^2))$ that satisfies
\begin{equation*}
	u = H_0 \quad \textrm{and} \quad \partial_\nu u = H_1 \quad \textrm{on}\ J \times \Gamma.
\end{equation*}
By construction this function satisfies the desired boundary conditions.
\end{proof}

\begin{remark}\label{rem_trace_curl}
For $\theta_0 = \frac{\pi}{2}$ we have $\cos \theta_0 = 0$ and $\sin \theta_0 = 1$
as well as $\tau_1 = - e_1$, $\nu_1 = - e_2$, $\tau_2 = e_2$ and $\nu_2 = -e_1$.
In this case the compatibility conditions in Corollary~\ref{cor_trace} read
\begin{equation*}
	\begin{array}{rcll}
		  \jump{h_1} & = & 0 & \qquad \textrm{in}\ J, \quad \textrm{if}\ p > 2, \\[0.25em]
		- \langle \partial_{x_1} h_0 \rangle_1 + \langle \partial_{x_2} h_0 \rangle_2 & = & \trace{h_1} & \qquad \textrm{in}\ J, \quad \textrm{if}\ p > 2,
	\end{array}
\end{equation*}
which explains the additional compatibility condition between $h_0$ and $h_1$ that is necessary in this case:
\begin{equation*}
	- \langle \partial_{x_1} h_0 \rangle_1 + \langle \partial_{x_2} h_0 \rangle_2
		= \langle \partial_{x_1} u_2 \rangle_1 - \langle \partial_{x_2} u_1 \rangle_2
		= \langle \text{curl}\,u \rangle_\bullet
		= \trace{h_1}, \quad \textrm{if}\ p > 2,
\end{equation*}
for every $u \in W^{1,p}(J, L^p(G, \R^2)) \cap L^p(J, W^{2,p}(G, \R^2))$
that satisfies $u \cdot \nu = h_0$ as well as $\text{curl}\,u = h_1$ on $J \times \Gamma$.
\end{remark}

The next auxiliary result is important,
since it allows for the inhomogeneous divergence constraint in problem \eqref{probl_perfect_slip}.

\begin{proposition}\label{prop_div_eq}
Let $J=(0,T)$ with $0 < T < \infty$ and let $G \subset \R^2$ be the wedge domain defined as in \eqref{def_wedge} with opening angle $\theta_0 \in (0, \pi)$
and let $\Gamma = \partial G \setminus \{ 0 \}$.
Assume that $p \in (1, \infty) \setminus \{ \frac{2 \theta_0}{3 \theta_0- \pi},\ \frac{2 \theta_0}{3 \theta_0 - 2 \pi},\ 2\,\}$.
Then for each
\begin{equation*}
	g \in W^{1,p}(J, \widehat{W}^{-1,p}(G)) \cap L^p(J, \widehat{W}^{1,p}(G))
\end{equation*}
there exists a function $u \in W^{1,p}(J, L^p(G, \R^2)) \cap L^p(J, W^{2,p}(G, \R^2))$ such that
\begin{equation}\label{prob_div}
	 	\begin{array}{r@{\ =\ }lll}
	 		 \div u &  g & \quad \text{in} & J \times G, \\[0.25em]
	 		 \text{curl} \ u =0, \ u \cdot \nu  & 0& \quad \text{on} & J \times \Gamma .\
	 	\end{array}
\end{equation}
\end{proposition}

\begin{proof}
Let $\phi \in L^p(J, K^3_p(G))$ be the unique solution of the problem
\begin{equation*} 
 	\begin{array}{r@{\ =\ }lll}
		\Delta \phi &  g &\text{in} & J \times G, \\[0.25em]
		\partial_\nu \phi  & 0& \text{on} & J \times \Gamma,
	\end{array}
\end{equation*}
which exists according to Corollary \ref{cor_probl_zeit}, since $g \in L^p(J, \widehat{W}^{1,p}(G))$.

By the fact that we also have $g \in  W^{1,p}(J, \widehat{W}^{-1,p}(G))$
it follows that $\phi$ is also a weak solution to the above problem, i.\,e.\ $\nabla \phi \in  W^{1,p}(J, L^p(G, \R^2))$.
Note that we have \linebreak $\partial^\alpha \phi \in L^p(J, L^p(G))$ for $|\alpha| = 3$, since $\phi \in L^p(J, K^3_p(G))$.
Now, let $u := \nabla \phi$.
We then have $u, \partial^\alpha u \in L^p(J, L^p(G, \R^2)) $ for $|\alpha| = 2$.
Interpolation (e.\,g.\ using the Gagliardo-Nirenberg inequality) yields that $\partial^\alpha u \in L^p(J, L^p(G, \R^2))$ also for $|\alpha|=1$.
Summarizing we have $u \in  W^{1,p}(J, L^p(G, \R^2)) \cap L^p(J, W^{2,p}(G, \R^2)).$
Moreover, $\mbox{div}\,u = \Delta \phi = g$ in $J \times G$ and $\mbox{curl}\,u = \mbox{curl}\,\nabla \phi = 0$ on $J \times \Gamma$.
Finally, $u \cdot \nu = \partial_\nu \phi = 0$ on $J \times \Gamma$.
\end{proof} 
 
Now, we are in position to prove the main result of this subsection.
 
\begin{theorem}\label{thm_perfect_slip}
Let $J=(0,T)$ with $0 < T < \infty$ and let $G \subset \R^2$ be the wedge domain defined as in \eqref{def_wedge} with opening angle $\theta_0 \in (0, \pi)$
and let $\Gamma = \partial G \setminus \{\,0\,\}$.
Assume that $p \in (1, \infty) \setminus \{ \frac{2 \theta_0}{3 \theta_0- \pi},\ \frac{2 \theta_0}{3 \theta_0 - 2 \pi},\ \frac{3}{2},\ 2,\ 3\,\}$.
Suppose the data satisfy the regularity condition
\begin{equation*}
(f,\ g,\ h_1,\ h_0,\ u_0) \in \mF
\end{equation*}
and the compatibility conditions
\begin{equation*}
	\begin{array}{rcll}
		\mbox{div}\ u_0 & = & g|_{t=0}, & \quad \textrm{if}\ p > 2, \\[0.5em]
		u_0 \cdot \nu & = & h_0|_{t=0}, & \quad \textrm{if}\ p> \frac{3}{2}, \\[0.5em]
		\mbox{curl}\ u_0 & = & h_1|_{t=0}, & \quad \textrm{if}\ p> 3,
	\end{array}
\end{equation*}
as well as
\begin{equation*}
	F(g, h_0) \in W^{1,p}(J, \widehat{W}^{-1,p}(G))
\end{equation*}
and $\jump{h_1} = 0$ in $J$, if $p > 2$, and
\begin{equation*}
	\langle \partial_{\tau_1} h_0 \rangle_1 + \langle \partial_{\tau_2} h_0 \rangle_2 = \trace{h_1} \quad \textrm{in}\ J,
		\qquad \textrm{if}\ \theta_0 = {\textstyle \frac{\pi}{2}}\ \textrm{and}\ p > 2.
\end{equation*}
Then there exists a unique solution $(u,p) \in \mE$ to problem \eqref{probl_perfect_slip}.
\end{theorem} 

\begin{proof}
The uniqueness of the solution $(u,p) \in \mE$ follows directly from \cite[Corollary~1]{Koehne-Saal-Westermann:Stokes-Wedge}.

To show the existence of the solution to \eqref{probl_perfect_slip} we proceed in three steps:
First, we employ Corollary~\ref{cor_trace} and choose $u_1 \in \mE_u$ such that
\begin{equation*}
	 	\begin{array}{r@{\ =\ }lll}
	 		\text{curl} \ u_1 &  h_1 & \quad \text{on} & J \times \Gamma, \\[0.25em]
	 		u_1 \cdot \nu &h_0 & \quad \text{on} & J \times \Gamma.
	 	\end{array}
\end{equation*}
Next, we employ Proposition~\ref{prop_div_eq} and choose $u_2 \in \mE_u$ such that
\begin{equation*}
	 	\begin{array}{r@{\ =\ }lll}
	 		 \div u_2 &  g - \div u_1 & \quad \text{in} & J \times G, \\[0.25em]
	 		 \text{curl} \ u_2 =0, \ u_2 \cdot \nu  & 0& \quad \text{on} & J \times \Gamma.
	 	\end{array}
\end{equation*}
Note that the compatibility conditions and the fact that $u_1 \cdot \nu = h_0$ on $J \times \Gamma$ ensure that
$g - \text{div}\,u_1 \in W^{1,p}(J, \widehat{W}^{-1,p}(G)) \cap L^p(J, W^{1,p}(G))$; cf.~Remark~\ref{rem_compat}.
Finally, we employ \mbox{\cite[Corollary~1]{Koehne-Saal-Westermann:Stokes-Wedge}} and choose $(u_3, p_3) \in \mE$ such that
\begin{equation*}
	 	\begin{array}{r@{\ =\ }lll}
	 		 \partial_t u_3 - \Delta u_3 + \nabla p &  f - \partial_t u_1 + \Delta u_1 - \partial_t u_2 + \Delta u_2 &
			\quad \text{in} & J \times G, \\[0.25em]
			\div u_3 & 0& \quad \text{in} & J \times G, \\[0.25em]
		    \text{curl} \ u_3 =0, \ u_3 \cdot \nu &0 &  \quad  \text{on} & J \times  \Gamma ,\\[0.25em]
			u_3(0) & u_0- u_1(0) - u_2(0)& \quad \text{in} & G.
	 	\end{array}
\end{equation*}
By construction $(u, p) := (u_1 + u_2 + u_3, p) \in \mE$ is a solution to \eqref{probl_perfect_slip}.
\end{proof}

\subsection{Inhomogeneous Free and Perfect Slip Boundary Conditions}

Let $\mE$ and $\mF$ be defined as in \eqref{solution_class} and \eqref{data_class}, respectively.
Here we consider the system \eqref{probu2perfectslip} and show that it is uniquely solvable within the maximal regularity class $\mE$.
Recall that the boundary of $G$ is decomposed as in \eqref{def_boundary} as $\partial G = \Gamma \cup \{\,0\,\}$
with its smooth part given as $\Gamma = \Gamma_1 \cup \Gamma_2$.
Also recall that $(\tau,\,\nu) = (\tau_j,\,\nu_j)$ for $j = 1,\,2$ denotes the positively oriented pair of unit tangential and unit outer normal vector on $\Gamma_j$ as introduced in Section~\ref{sec_intro}.

For the boundary conditions in problem \eqref{probu2perfectslip} we observe that
\begin{align*}
\tau^T D_\pm(u)\nu  &= \frac{1}{2} \begin{pmatrix}
\partial_{x_1} u_1 \pm  \partial_{x_1} u_1 &  \partial_{x_1} u_2 \pm  \partial_{x_2} u_1 \\ \partial_{x_2} u_1 \pm  \partial_{x_1} u_2   & \partial_{x_2} u_2 \pm  \partial_{x_2} u_2
\end{pmatrix} \nu \cdot \tau \\[0.5em]
& = \frac{1}{2} \begin{pmatrix}
\partial_{x_1} (u \cdot \nu) \\ \partial_{x_2} (u \cdot \nu)
\end{pmatrix} \cdot \tau \pm \frac{1}{2} \begin{pmatrix}
\partial_\nu u_1 \\ \partial_\nu u_2
\end{pmatrix} \cdot \tau \\[0.5em]
& = {\textstyle \frac{1}{2}} \partial_\tau (u \cdot \nu) \pm {\textstyle \frac{1}{2}} \partial_\nu (u \cdot \tau) \qquad \textrm{on}\ J \times \Gamma,
\end{align*}
which implies that
\begin{equation*}
	\begin{array}{rcl}
		\tau^T D_+(u)\nu & = & {\textstyle \frac{1}{2}} \partial_\tau (u \cdot \nu) + {\textstyle \frac{1}{2}} \partial_\nu (u \cdot \tau) \\[0.25em]
			& = & \partial_\tau (u \cdot \nu) - {\textstyle \frac{1}{2}} \partial_\tau (u \cdot \nu) + {\textstyle \frac{1}{2}} \partial_\nu (u \cdot \tau) = \partial_\tau (u \cdot \nu) - {\textstyle \frac{1}{2}} \textrm{curl}\,u \qquad \textrm{on}\ J \times \Gamma
	\end{array}
\end{equation*}
as well as
\begin{equation*}
	\tau^T D_-(u)\nu = {\textstyle \frac{1}{2}} \partial_\tau (u \cdot \nu) - {\textstyle \frac{1}{2}} \partial_\nu (u \cdot \tau)
		= {\textstyle \frac{1}{2}} \textrm{curl}\,u \qquad \textrm{on}\ J \times \Gamma.
\end{equation*}
Therefore, if the tangential boundary condition in \eqref{probu2perfectslip} is posed based on $D_+$,
then \eqref{probu2perfectslip} is equivalent to
\begin{equation}\label{probu2curlformmain}
	 	\begin{array}{r@{\ =\ }lll}
	 		 \partial_t u - \Delta u + \nabla p & f &
			\quad \text{in} & J \times G, \\[0.25em]
			\div u & g & \quad \text{in } & J \times G, \\[0.25em]
	 		u \cdot \nu & h_0 & \quad \text{on} & J \times \Gamma,\\[0.25em]
			u(0) & u_0& \quad \text{in} & G
	 	\end{array}
\end{equation}
together with the boundary condition
\begin{equation}\label{probu2curlformplus}
	\textrm{curl}\,u = 2 (\partial_\tau h_0 + h_1) \quad \text{on}\ J \times \Gamma.
\end{equation}
Analogously, if the tangential boundary condition in problem \eqref{probu2perfectslip} is posed based on $D_-$,
then \eqref{probu2perfectslip} is equivalent to \eqref{probu2curlformmain}
together with the boundary condition
\begin{equation}\label{probu2curlformminus}
	\textrm{curl}\,u = - 2 h_1 \quad \text{on}\ J \times \Gamma.
\end{equation}
Both systems (\ref{probu2curlformmain}, \ref{probu2curlformplus}) and (\ref{probu2curlformmain}, \ref{probu2curlformminus})
are uniquely solvable using Theorem~\ref{thm_perfect_slip} and, hence, we obtain the following result.
\begin{corollary}\label{cor_free_perfect_slip}
Let $J=(0,T)$ with $0 < T < \infty$ and let $G \subset \R^2$ be the wedge domain defined as in \eqref{def_wedge} with opening angle $\theta_0 \in (0, \pi)$
and let $\Gamma = \partial G \setminus \{\,0\,\}$.
Assume that $p \in (1, \infty) \setminus \{ \frac{2 \theta_0}{3 \theta_0- \pi},\ \frac{2 \theta_0}{3 \theta_0 - 2 \pi},\ \frac{3}{2},\ 2,\ 3\,\}$.
Suppose the data satisfy the regularity condition
\begin{equation*}
(f,\ g,\ h_1,\ h_0,\ u_0) \in \mF
\end{equation*}
and the compatibility conditions
\begin{equation*}
	\begin{array}{rcll}
		\mbox{div}\ u_0 & = & g|_{t=0}, & \quad \textrm{if}\ p > 2, \\[0.5em]
		u_0 \cdot \nu & = & h_0|_{t=0}, & \quad \textrm{if}\ p> \frac{3}{2}, \\[0.5em]
		- \tau^T D_\pm(u_0) \nu & = & h_1|_{t=0}, & \quad \textrm{if}\ p> 3,
	\end{array}
\end{equation*}
as well as
\begin{equation*}
	F(g, h_0) \in W^{1,p}(J, \widehat{W}^{-1,p}(G)).
\end{equation*}
If the boundary condition is posed based on $D_+$,
then assume the compatibility conditions $\jump{\partial_\tau h_0 + h_1} = 0$ in $J$, if $p > 2$, and
\begin{equation*}
	{\textstyle \frac{1}{2}} \langle \partial_{\tau_1} h_0 \rangle_1 + {\textstyle \frac{1}{2}} \langle \partial_{\tau_2} h_0 \rangle_2 = \trace{\partial_\tau h_0 + h_1} \quad \textrm{in}\ J,
		\qquad \textrm{if}\ \theta_0 = {\textstyle \frac{\pi}{2}}\ \textrm{and}\ p > 2.
\end{equation*}
If the boundary condition is posed based on $D_-$,
then assume the compatibility conditions $\jump{h_1} = 0$ in $J$, if $p > 2$, and
\begin{equation*}
	- {\textstyle \frac{1}{2}} \langle \partial_{\tau_1} h_0 \rangle_1 - {\textstyle \frac{1}{2}} \langle \partial_{\tau_2} h_0 \rangle_2 = \trace{h_1} \quad \textrm{in}\ J,
		\qquad \textrm{if}\ \theta_0 = {\textstyle \frac{\pi}{2}}\ \textrm{and}\ p > 2.
\end{equation*}
Then there exists a unique solution $(u,p) \in \mE$ to problem \eqref{probu2perfectslip}.
\end{corollary}

\subsection{Proof of Theorem \ref{thm_navierslip}}

A unique solution to \eqref{problem_navierslipbc} can now be obtained with the aid of the usual perturbation argument.
To this end, we denote by $L: \mE \longrightarrow \mF$ the linear operator induced by the left-hand side of problem \eqref{probu2perfectslip}.
Now, if $(f, g, h_1, h_0, u_0) \in \mF$ satisfy all compatibility conditions stated in Theorem~\ref{thm_navierslip},
then \eqref{problem_navierslipbc} is equivalent to
\begin{equation*}
	L(u, p) = (f, g, h_1 - \alpha u \cdot \tau, h_0 , u_0).
\end{equation*}
Now, we choose $\widetilde{h}_1 \in \mF_\tau$ such that $\trace{\widetilde{h}_1} = 0$, if $p > 2$
and $\widetilde{h}_1|_{t = 0} = \alpha u_0|_\Gamma \cdot \tau$.
This is possible, since $\trace{\alpha u_0 \cdot \tau} = 0$, if $p > 2$, due to the requirement $\trace{\alpha} = 0$.
By construction, the data $(f, g, h_1 - \widetilde{h}_1, h_0, u_0) \in \mF$ satisfy all compatibility conditions stated in Corollary~\ref{cor_free_perfect_slip}.

Hence, Corollary~\ref{cor_free_perfect_slip} shows that there exists a unique solution $(u_\ast, p_\ast)$
to the problem $L(u_\ast, p_\ast) = (f, g, h_1 - \widetilde{h}_1, h_0, u_0)$.
Thus, the ansatz $(u, p) = (u_\ast, p_\ast) + (v, q)$ leads to the problem
\begin{equation*}
	\mL(v, q) = \widetilde{h}_1 - \alpha u_\ast|_\Gamma \cdot \tau - \alpha v|_\Gamma \cdot \tau, \qquad \qquad (v, q) \in {}_0 \mE,
\end{equation*}
where the linear operator $\mL: {}_0 \mE \longrightarrow {}_0 \mF_\tau$ between the spaces
\begin{equation*}
	{}_0 \mE := \left\{\,(w, r) \in \mE\,: \begin{array}{c} \partial_t w - \Delta w + \nabla r = 0\ \textrm{in}\ J \times G,\ \textrm{div}\,w = 0\ \textrm{in}\ J \times G \\[0.5em] w \cdot \nu = 0\ \textrm{on}\ J \times \Gamma,\ w|_{t = 0} = 0\ \textrm{in}\ G \end{array} \right\}
\end{equation*}
and
\begin{equation*}
	{}_0 \mF_\tau := \Big\{\,h \in \mF_\tau\,:\,\trace{h} = 0,\ \textrm{if}\ p > 2,\ h|_{t = 0} = 0,\ \textrm{if}\ p > 3\,\Big\}
\end{equation*}
is given as $\mL(w, r) := (\textrm{curl}\,w)|_\Gamma$ for $(w, r) \in {}_0 \mE$.

Thanks to the homogeneous initial conditions
the operator $\mL$ is a linear isomorphism by Corollary~\ref{cor_free_perfect_slip},
where the operator norm of $\mL^{-1}$ does not depend on the length $T > 0$ of the time interval $J = (0,T)$ under consideration.
Moreover, we have
\begin{equation*}
	\trace{\widetilde{h}_1 - \alpha u_\ast \cdot \tau} = \trace{\alpha v \cdot \tau} = 0, \quad \textrm{if}\ p > 2,
\end{equation*}
since $\trace{\alpha} = \trace{\widetilde{h}_1} = 0$, as well as
\begin{equation*}
	(\widetilde{h}_1 - \alpha u_\ast|_\Gamma \cdot \tau)|_{t = 0}
		= \widetilde{h}_1|_{t = 0} - \alpha u_0|_\Gamma \cdot \tau = 0, \quad \textrm{if}\ p > 3,
\end{equation*}
which shows that $\widetilde{h}_1 - \alpha u_\ast|_\Gamma \cdot \tau \in {}_0 \mF_\tau$.
Clearly, we also have $\alpha v|_\Gamma \cdot \tau \in {}_0 \mF_\tau$ for all $(v, q) \in {}_0 \mE$
and we are left with the task to solve the problem
\begin{equation*}
	(1 - \mL^{-1} R)(v, q) = \mL^{-1}(\widetilde{h}_1 - \alpha u_\ast|_\Gamma \cdot \tau), \qquad \qquad (v, q) \in {}_0 \mE,
\end{equation*}
where the linear operator $R: {}_0 \mE \longrightarrow {}_0 \mF_\tau$ is given as
$R(v, q) := - \alpha v|_\Gamma \cdot \tau$.
However, this operator is of lower order and the usual estimates employed within perturbation arguments for parabolic problems
show that $1 - \mL^{-1} R$ is invertible by a Neumann series argument, at least for small values $T < T^\ast$.
Here $T^\ast > 0$ is independent of the data.
Consequently, problem \eqref{problem_navierslipbc} may be solved this way successively on small time intervals,
which cover any given time interval $J = (0, T)$ after finitely many steps.
This completes the proof of Theorem~\ref{thm_navierslip}.
\qed
\appendix 

\section{Hardy's inequality on the 2D wedge domain}
\label{sec_hardy}
The famous Hardy's inequality is well known and many proofs exist within the literature.
However, a proper formulation for the wedge requires boundary conditions at the corner point $x = 0$
or at infinity, if one wants to have a version of Hardy's inequality at hand,
that is not only valid for equivalence classes of functions that differ by additive constants.
A version in the latter sense is easily deduced as a consequence of, for instance, \cite[Corollary~VIII.1.5.3]{Amann:Parabolic-Problems-2}.

However, we note that $L^p_\gamma(0, \infty) \hookrightarrow L^1_{\textrm{\upshape loc}}(0, \infty)$ for $1 < p < \infty$ by H{\"o}lder's inequality,
provided that $\gamma \in \R$ with $\gamma < p - 1$.
Hence, if $u \in L^p_{\textrm{\upshape loc}}(0, \infty)$ and $u^\prime \in L_\gamma^p(0, \infty)$,
then $u \in W^{1, 1}_{\textrm{\upshape loc}}(0, \infty)$, if $\gamma < p - 1$,
which shows that the trace $u(0)$ is well-defined in this case.
Analogously, the value $u(\infty) = \lim_{x \rightarrow \infty} u(x)$ is well-defined, if $\gamma > p - 1$.
Indeed, if $u \in L^p_{\textrm{\upshape loc}}(0, \infty)$ and $u^\prime \in L_\gamma^p(0, \infty)$ with $\gamma > p - 1$,
then for $v: (0,\,\infty) \longrightarrow \C$ given as $v(x) := u(\frac{1}{x})$ for $x > 0$ we obtain $v \in L^p_{\textrm{\upshape loc}}(0, \infty)$
as well as
\begin{equation*}
	\|v^\prime\|^p_{L^p_\gamma(0, a)}
		= \int^a_0 |v^\prime(x)|^p x^\gamma d(x)
		= \int^\infty_{1 / a} |u^\prime(y)|^p y^{2 (p - 1) - \gamma} d(y)
		\leq \|u^\prime\|^p_{L^p_\gamma(0, \infty)}
\end{equation*}
for $0 < a < 1$, where we used that $2 (p - 1) - \gamma < \gamma$,
i.\,e.\ $v \in W^{1, 1}_{\textrm{\upshape loc}}(0, a)$
and $u(\infty) := v(0)$ is well-defined.
Now, with the same proof as given in \cite{Amann:Parabolic-Problems-2} we obtain the following version of \cite[Corollary~VIII.1.5.3]{Amann:Parabolic-Problems-2},
which is Hardy's inequality on the halfline $(0, \infty)$.

\begin{corollary}\label{lem_hardy_amann}
Suppose $1<p< \infty$, $\gamma \in \R$ and $\gamma \neq p-1$.
Let $u \in L^p_{\textrm{\upshape loc}}(0, \infty)$ with $u' \in L_\gamma^p(0, \infty)$
such that $u(0) = 0$, if $\gamma < p - 1$, and $u(\infty) = 0$, if $\gamma > p - 1$, respectively.
Then we have
\begin{equation*}
	\left\| \frac{u}{x} \right\|_{L_{\gamma}^p(0, \infty)} \leq C(p, \gamma)\,\| u' \|_{L_\gamma^p((0, \infty))}
\end{equation*}
with a constant $C(p, \gamma) > 0$ that is independent of $u$.
\end{corollary}

In the following let $\psi:= \psi_p  \circ \psi_E : \Omega \rightarrow G$ be the transformation from the wedge onto the layer domain defined at the beginning of Section 3.
As consequences of Corollary~\ref{lem_hardy_amann} we obtain:

\begin{lemma}\label{lemma_hardy}
Let $1<p< \infty$, $\gamma \in \R$ such that $\gamma \neq p-2$ and $\rho:= |(x_1, x_2)|$.
Let $u \in L_{\textrm{\upshape loc}}^p(G)$ with $\nabla u \in L^p_\gamma(G)$
such that $u(0) = 0$, if $\gamma < p - 2$, and $u(\infty) = 0$, if $\gamma > p - 2$, respectively.
Then we have
\begin{equation*}
\| \rho^{-1} u \|_{L_\gamma^p(G)} \leq C(p, \gamma)\,\| \nabla u \|_{L_\gamma^p(G)}
\end{equation*}
with a constant $C(p, \gamma) > 0$ that is independent of $u$.
\end{lemma}

\begin{proof}
Let $\tilde{\gamma} \in \R$ such that $\tilde{\gamma} \neq p-1$.
Let $v \in L^p_{\textrm{\upshape loc}}(0, \infty)$ with $ v' \in L^p_{\tilde{\gamma}}(0, \infty)$
such that $v(0) = 0$, if $\tilde{\gamma} < p - 1$, and $v(\infty) = 0$, if $\tilde{\gamma} > p - 1$, respectively.
Then by Lemma~\ref{lem_hardy_amann} we have that
\begin{align}\label{ungl_hardy}
\int_\R e^{(\tilde{\gamma} -p)x} |v (e^x)|^p e^x dx &= \int_{0}^\infty \left| \frac{v(y)}{y} \right|^p y^{\tilde{\gamma}} dy  \leq C(p, \tilde{\gamma}) \int_{0}^\infty | v'(y)|^p y^{ \tilde{\gamma}} dy \notag \\[0.5em]
& =C(p, \tilde{\gamma})\int_\R e^{\tilde{\gamma}x} |v'(e^x)|^p e^x dx.
\end{align}

Now, let $ u \in L^p_{\textrm{\upshape loc}}(G) $ with $\nabla u \in L_\gamma^p(G)$
such that $u(0) = 0$, if $\gamma < p - 2$, and $u(\infty) = 0$, if $\gamma > p - 2$, respectively, and set $\tilde{\gamma} := \gamma + 1$.
Then the above calculation implies
\begin{samepage}
\begin{align*}
\| \rho^{-1} u \|^p_{L^p_\gamma(G)} &= \int_{G} \left|\frac{u(x_1, x_2)}{\rho(x_1, x_2)} \right|^p \rho^{\gamma} d(x_1, x_2) = \int_0^{\theta_0} \int_\R e^{(\gamma -p) x} | u(\psi (x, \theta))|^p e^{2x} dx  d \theta \\[0.5em]
&= \int_0^{\theta_0} \int_\R e^{(\gamma-p + 1) x} |u(\psi (x, \theta))|^p e^x dx d\theta \\[0.5em]
& \leq  C(p, \gamma + 1) \int_0^{\theta_0} \int_\R e^{(\gamma + 1 )x} | \nabla u(\psi(x, \theta)) |^p e^x dx d\theta\\[0.5em]
& =C(p, \gamma +1 ) \int_{0}^{\theta_0} \int_\R e^{\gamma  x} | \nabla u(\psi (x, \theta))|^p e^{2x} dx d\theta \\[0.5em]
& = C(p, \gamma + 1 ) \int_G \rho^{\gamma} | \nabla u(x_1, x_2) |^p d(x_1, x_2) \\[1.0em]
&= C(p, \gamma + 1 )\,\| \nabla u \|^p_{L_\gamma^p(G)}.
\end{align*}
\end{samepage}
Note that $\tilde{\gamma} \gtrless p - 1$, if and only if $\gamma \gtrless p - 2$.
\end{proof}	

\begin{lemma}
Let $1<p< \infty$, $\gamma \in \R$ such that $\gamma \neq -2$ and $\rho:= |(x_1, x_2)|$.
Let $u \in L_{\textrm{\upshape loc}}^p(G)$ with $\rho \nabla u \in L^p_\gamma(G)$
such that $u(0) = 0$, if $\gamma < - 2$, and $u(\infty) = 0$, if $\gamma > - 2$, respectively.
Then we have
\begin{equation*}
\| u \|_{L^p_\gamma(G)} \leq C(p, \gamma)\,\| \rho \nabla u \|_{L^p_\gamma(G)}
\end{equation*}
with a constant $C(p, \gamma) > 0$ that is independent of $u$.
\end{lemma}

\begin{proof}
Let $\tilde{\gamma} \in \R$ such that $\tilde{\gamma} \neq p-1$.
Let $v \in L^p_{\textrm{\upshape loc}}(0, \infty)$ with $ v' \in L^p_{\tilde{\gamma}}(0, \infty)$
such that $v(0) = 0$, if $\tilde{\gamma} < p - 1$, and $v(\infty) = 0$, if $\tilde{\gamma} > p - 1$, respectively.
Then as above by Lemma~\ref{lem_hardy_amann} we obtain \eqref{ungl_hardy}.

Now, let $ u \in L^p_{\textrm{\upshape loc}}(G) $ with $\rho \nabla u \in L^p_\gamma(G)$
i.\,e.\ $\rho^{1 + \frac{\gamma}{p}}\nabla u \in L^p(G)$,
such that $u(0) = 0$, if $\gamma < - 2$, and $u(\infty) = 0$, if $\gamma > - 2$, respectively,
and set $\tilde{\gamma} := \gamma + p + 1$.
Then the above calculation implies
\begin{align*}
\| u \|^p_{L^p_\gamma(G)} &= \int_{G} |u(x_1, x_2)|^p \rho^{\gamma} d(x_1, x_2) = \int_0^{\theta_0} \int_\R e^{\gamma x} | u(\psi (x, \theta))|^p e^{2x} dx  d \theta \\[0.5em]
&= \int_0^{\theta_0} \int_\R e^{(\gamma + 1) x} |u(\psi (x, \theta))|^p e^x dx d\theta \\[0.5em]
& \leq  C(p, \gamma + p + 1) \int_0^{\theta_0} \int_\R e^{(\gamma + p + 1)x} | \nabla u(\psi(x, \theta)) |^p e^x dx d\theta\\[0.5em]
& =C(p, \gamma + p + 1) \int_{0}^{\theta_0} \int_\R e^{(\gamma + p) x} | \nabla u(\psi (x, \theta))|^p e^{2x} dx d\theta \\[0.5em]
& = C(p, \gamma + p + 1) \int_G \rho^{\gamma+ p} | \nabla u(x_1, x_2) |^p d(x_1, x_2) \\[1.0em]
&= C(p, \gamma + p + 1)\,\| \rho \nabla u \|^p_{L_\gamma^p(G)}.
\end{align*}
Note that $\tilde{\gamma} \gtrless p - 1$, if and only if $\gamma \gtrless - 2$.
\end{proof}	

\section{Density of Smooth Functions in $\widehat{W}^{m, p}(\Omega)$}
\label{sec_density}
In this section we prove that the space of smooth, compactly supported functions
\begin{equation*}
	C^\infty_c(\bar{\Omega}) := \big\{\,\phi|_\Omega\,:\,\phi \in C^\infty_c(\R^n)\,\big\}
\end{equation*}
is dense in the homogeneous Sobolev space $\widehat{W}^{m, p}(\Omega)$ for a large class of domains $\Omega \subseteq \R^n$ with $n \in \N$.
Our approach relies on the availability of Poincar{\'e}'s inequality for a certain family of subdomains of $\Omega$
and the controllability of the constants that appear in these inequalities.
To this end, we denote for $x \in \R^n$ and $\varepsilon > 0$ by $R_\varepsilon(x) := B_\varepsilon(x) \setminus \bar{B}_{\varepsilon / 2}(x)$
the open ring with inner and outer radius $\varepsilon / 2$ and $\varepsilon$, respectively.
Given $x \in \R^n$ and $k \in \N$ we denote by $\varphi_k(\,\cdot\,,\,x): \R^n \longrightarrow \R^n$
\begin{equation*}
	\varphi_k(y,\,x) := x + k(y - x), \qquad \qquad y \in \R^n,
\end{equation*}
a dilatation w.\,r.\,t.\ $x$.
Now, we say that a domain $\Omega \subseteq \R^n$ has the {\itshape uniform dilatation property},
if there exist $x \in \R^n$ and $r > 0$ such that
\begin{equation*}
	\varphi_k(\Omega \cap R_r(x),\,x) = \Omega \cap R_{k r}(x) \quad \textrm{for all}\ k \in \N.
\end{equation*}
For instance, if $G \subseteq \R^2$ is a wedge domain defined in \eqref{def_wedge} with opening angle $\theta_0 \in (0,\,2 \pi)$
and $\Omega \subseteq \R^2$ is a domain such that $\Omega \setminus \bar{B}_r(0) = G \setminus \bar{B}_r(0)$ for some $r > 0$,
then $\Omega$ has the uniform dilatation property (for $x = 0$).
In particular, $G$ has the uniform dilatation property (for $x = 0$ and every $r > 0$).
More generally, if $\Omega_1,\,\dots,\,\Omega_m \subseteq \R^2$ are domains such that $\Omega_k \setminus \bar{B}_r(0) = G_k \setminus \bar{B}_r(0)$ are pairwise disjoint
for some $r > 0$ and (rotated) wedge domains $G_1,\,\dots,\,G_m \subseteq \R^2$, then $\Omega = \Omega_1 \cup \ldots \cup \Omega_m$ has the uniform dilatation property.
This includes a certain class of aperture domains.
Other examples for domains that have the uniform dilatation property are given by perturbed cones in $\R^n$.
Now, the importance of this class of domains stems from the following generalized version of Poincar{\'e}'s inequality.

\begin{lemma}
\label{Uniform-Dilatation-Domain-Poincare}
Let $n \in \N$ and let $\Omega \subseteq \R^n$ be a domain.
Moreover, let $U \in \{\,B,\,R\,\}$, let $x \in \R^n$ and let $r > 0$ such that $\varphi_k(\Omega \cap U_r(x),\,x) = \Omega \cap U_{k r}(x)$ for all $k \in \N$.
Assume that $\Omega \cap U_r(x) \neq \varnothing$ is a domain that has the cone property.
Then for every $1 \leq p < \infty$ there exists a constant $C = C(p) > 0$ such that
\begin{equation*}
	\|u - \mu(u,\,\Omega \cap U_{k r}(x))\|_{L^p(\Omega\,\cap\,U_{k r}(x))} \leq k C \|\nabla u\|_{L^p(\Omega\,\cap\,U_{k r}(x))}, \quad u \in W^{1, p}(\Omega \cap U_{k r}(x)),
\end{equation*}
for all $k \in \N$.
Here, $\mu(v,\,\Omega \cap U_{k r}(x)) := |\Omega \cap U_{k r}(x)|^{-1} \int_{\Omega\,\cap\,U_{k r}(x)} v(y)\,\mbox{d}y$
denotes the mean value of a function $v \in L^1(\Omega \cap U_{k r}(x))$ for $k \in \N$.
\end{lemma}
\begin{proof}
We fix $1 \leq p < \infty$.
Since $\Omega \cap U_r(x) \neq \varnothing$ is a domain that has the cone property,
it follows from the Rellich-Kondrachov theorem, see e.\,g.\ \cite[Thm.\ 6.2]{Adams-Fournier:Sobolev-Spaces},
that the inclusion $W^{1, p}(\Omega \cap U_r(x)) \subseteq L^p(\Omega \cap U_r(x))$ is compact.
Hence, Poincar{\'e}'s inequality is available and there exists a constant $C = C(p) > 0$
such that
\begin{equation*}
	\|v - \mu(v,\,\Omega \cap U_r(x))\|_{L^p(\Omega\,\cap\,U_r(x))} \leq C \|\nabla v\|_{L^p(\Omega\,\cap\,U_r(x))}
\end{equation*}
for all $v \in W^{1, p}(\Omega \cap U_r(x))$.
Now, fix $k \in \N$ and $u \in W^{1, p}(\Omega \cap U_{k r}(x))$. \pagebreak
Then we obviously have $v := u \circ \varphi_k(\,\cdot\,,\,x) \in W^{1, p}(\Omega \cap U_r(x))$.
Moreover, we have
\begin{equation*}
	\begin{array}{l}
		\mu(v,\,\Omega \cap U_r(x)) \cdot |\Omega \cap U_r(x)|
			= {\displaystyle \!\!\!\!\!\!\!\! \int\limits_{\Omega\,\cap\,U_r(x)} \!\!\!\!\!\!\!\! v(y)\,\mbox{d}y} \\[2.5em]
			\qquad = {\displaystyle k^{- n} \!\!\!\!\!\!\!\!\! \int\limits_{\Omega\,\cap\,U_{k r}(x)} \!\!\!\!\!\!\!\!\! w(y)\,\mbox{d}y}
			= k^{- n} \mu(u,\,\Omega \cap U_{k r}(x)) \cdot |\Omega \cap U_{k r}(x)|,
	\end{array}
\end{equation*}
which implies that $\mu(v,\,\Omega \cap U_r(x)) = \mu(u,\,\Omega \cap U_{k r}(x))$ in view of the scaling property
$|\Omega \cap U_r(x)| = k^{- n} |\Omega \cap U_{k r}(x)|$,
as well as
\begin{equation*}
	\begin{array}{l}
		\|v - \mu(v,\,\Omega \cap U_r(x))\|^p_{L^p(\Omega\,\cap\,U_r(x))}
			= {\displaystyle \!\!\!\!\!\!\!\! \int\limits_{\Omega\,\cap\,U_r(x)} \!\!\!\!\!\!\! |v(y) - \mu(v,\,\Omega \cap U_r(x))|^p\,\mbox{d}y} \\[2.5em]
			\qquad = {\displaystyle k^{- n} \!\!\!\!\!\!\!\!\! \int\limits_{\Omega\,\cap\,U_{k r}(x)} \!\!\!\!\!\!\!\! |w(y) - \mu(v,\,\Omega \cap U_r(x))|^p\,\mbox{d}y}
			= k^{- n} \|u - \mu(u,\,\Omega \cap U_{k r}(x))\|^p_{L^p(\Omega\,\cap\,U_{k r}(x))}
	\end{array}
\end{equation*}
and
\begin{equation*}
	\|\nabla v\|^p_{L^p(\Omega\,\cap\,U_r(x))}
		= \!\!\!\!\!\!\!\! \int\limits_{\Omega\,\cap\,U_r(x)} \!\!\!\!\!\!\! |\nabla v(y)|^p\,\mbox{d}y
		= k^{p - n} \!\!\!\!\!\!\!\!\! \int\limits_{\Omega\,\cap\,U_{k r}(x)} \!\!\!\!\!\!\!\! |\nabla w(y)|^p\,\mbox{d}y
		= k^{p - n} \|\nabla u\|^p_{L^p(\Omega\,\cap\,U_{k r}(x))}.
\end{equation*}
This yields the desired estimate.
\end{proof}

In order to control lower order by higher order derivatives,
we need to adjust the functions that shall be estimated not only by an additive constant,
but by a polynomial of a certain degree.
To this end, we fix $m \in \N$ and a bounded, open set $U \subseteq \R^n$.
Now, we fix $u \in W^{m, p}(U)$ and we define $p_0[u] := \mu(u,\,U)$.
Given $0 < \ell \leq m$ we further define $q_\ell[u](x) := \sum_{|\alpha| = \ell} \frac{1}{\alpha!} \mu(\partial^\alpha u,\,U) x^\alpha$ for $x \in \R^n$
and $p_\ell[u] := q_\ell[u] + p_{\ell - 1}[u - q_\ell[u]]$.
This way we recursively define a family of polynomials $p_0[u],\,\dots,\,p_m[u]$,
where $p_\ell[u]$ is of degree $\ell$ for $0 \leq \ell \leq m$ and
\begin{equation*}
	\mu(\partial^\alpha(u - p_\ell[u]),\,U) = 0, \qquad |\alpha| \leq \ell, \qquad 0 \leq \ell \leq m.
\end{equation*}
Indeed, the assertion is obvious for $\ell = 0$.
For $0 < \ell \leq m$ and $|\alpha| = \ell$ we have
\begin{equation*}
	\partial^\alpha(u - p_\ell[u]) = \partial^\alpha(u - q_\ell[u]) = \partial^\alpha u - \mu(\partial^\alpha u,\,U),
\end{equation*}
which yields the assertion.
Finally, for $0 < \ell \leq m$ and $|\alpha| < \ell$ we have
\begin{equation*}
	\partial^\alpha(u - p_\ell[u])
		= \partial^\alpha((u - q_\ell[u]) - p_{\ell - 1}[u - q_\ell[u]]) = \partial^\alpha(v - p_{\ell - 1}[v])
\end{equation*}
with $v := u - q_\ell[u]$, so that we obtain the assertion by an induction w.\,r.\,t.\ $\ell$.
Note that we have $p_0,\,\dots,\,p_{m - 1} \in \widehat{W}^{m, p}(\R^n)$.
Since the construction of $p_0,\,\dots,\,p_m$ depends on $U$,
we also write $p_\ell[u,\,U]$ instead of $p_\ell[u]$, if we want to emphasize, which bounded, open set the construction is based on.
Now, we obtain the following generalization of Lemma~\ref{Uniform-Dilatation-Domain-Poincare} to higher order derivatives.

\begin{lemma}
\label{Uniform-Dilatation-Domain-Poincare-Higher-Order}
Let $n \in \N$ and let $\Omega \subseteq \R^n$ be a domain.
Moreover, let $U \in \{\,B,\,R\,\}$, let $x \in \R^n$ and let $r > 0$ such that $\varphi_k(\Omega \cap U_r(x),\,x) = \Omega \cap U_{k r}(x)$ for all $k \in \N$.
Assume that $\Omega \cap U_r(x) \neq \varnothing$ is a domain that has the cone property.
Then for every $m \in \N$ and every $1 \leq p < \infty$ there exists a constant $C = C(m,\,n,\,p) > 0$ such that
\begin{equation*}
	\|\partial^\alpha(u - p_{m - 1}[u,\,\Omega \cap U_{k r}(x)])\|_{L^p(\Omega\,\cap\,U_{k r}(x))} \leq k^{m - |\alpha|} C \|\nabla^m u\|_{L^p(\Omega\,\cap\,U_{k r}(x))}
\end{equation*}
for all $u \in W^{m, p}(\Omega \cap U_{k r}(x))$, all $k \in \N$ and all $\alpha \in \N^n_0$ with $|\alpha| < m$.
\end{lemma}
\begin{proof}
We fix $k \in \N$ and $u \in W^{m, p}(\Omega\,\cap\,U_{k r}(x))$ as well as $\alpha \in \N^n_0$ with $|\alpha| < m$.
Moreover, we set $q := p_{m - 1}[u,\,\Omega \cap U_{k r}(x)]$ and note that $q$ is a polynomial of degree $m - 1$ such that
\begin{equation*}
	\mu(\partial^\beta (u - q),\,\Omega \cap U_{k r}(x)) = 0, \qquad \qquad |\beta| < m.
\end{equation*}
Now, Lemma~\ref{Uniform-Dilatation-Domain-Poincare} yields
\begin{equation*}
	\|\partial^\alpha(u - q)\|_{L^p(\Omega\,\cap\,U_{k r}(x))}
		\leq k C \|\nabla \partial^\alpha(u - q)\|_{L^p(\Omega\,\cap\,U_{k r}(x))}
\end{equation*}
and if $|\alpha| = m - 1$, we have $\nabla \partial^\alpha q = 0$ so that we obtain the desired estimate.
If $|\alpha| < m - 1$, then $\beta_j := \alpha + e_j$ satisfies $|\beta_j| = |\alpha| + 1 < m$ for $j = 1,\,\dots,\,n$
and we apply Lemma~\ref{Uniform-Dilatation-Domain-Poincare} again to obtain
\begin{equation*}
	\|\partial^{\beta_j}(u - q)\|_{L^p(\Omega\,\cap\,U_{k r}(x))}
		\leq k C \|\nabla \partial^{\beta_j}(u - q)\|_{L^p(\Omega\,\cap\,U_{k r}(x))}, \qquad j = 1,\,\dots,\,n,
\end{equation*}
which yields
\begin{equation*}
	\|\partial^\alpha(u - q)\|_{L^p(\Omega\,\cap\,U_{k r}(x))}
		\leq k^2 C^\prime \|\nabla^2 \partial^\alpha(u - q)\|_{L^p(\Omega\,\cap\,U_{k r}(x))}.
\end{equation*}
Hence, if $|\alpha| = m - 2$, we obtain the desired estimate.
Otherwise, if $|\alpha| < m - 2$, we apply Lemma~\ref{Uniform-Dilatation-Domain-Poincare} again.
Then, clearly, after finitely many steps we arrive at the desired estimate.
Note that $\nabla^m q = 0$ so that the polynomial $q$ does not appear on the right-hand side within the final estimate.
\end{proof}

For the proof of Lemma~\ref{Uniform-Dilatation-Domain-Poincare} and Lemma~\ref{Uniform-Dilatation-Domain-Poincare-Higher-Order} the cone property played an important role.
For our main result we even need a slightly stronger assumption.
Thus, we say that a domain $\Omega \subseteq \R^n$ has the {\itshape regular, uniform dilatation property},
if there exist $x \in \R^n$ and $r > 0$ such that
\begin{enumerate}[(i)]
	\item $\Omega$ has the segment property,
	\item we have $\varphi_k(\Omega \cap R_r(x),\,x) = \Omega \cap R_{k r}(x)$ for all $k \in \N$,
	\item $\Omega \cap R_r(x) \neq \varnothing$ is a domain that has the cone property, and
	\item $\Omega \cap B_{k r}(x) \neq \varnothing$ is a (bounded) domain with a locally Lipschitz continuous boundary for all $k \in \N$.
\end{enumerate}
Clearly, the wedge domain $G \subseteq \R^2$ with arbitrary opening angle $\theta_0 \in (0,\,2 \pi)$ defined as in (\ref{def_wedge}) 
has the regular, uniform dilatation property (for $x = 0$ and every $r > 0$).
Now, we obtain the following approximation result.

\begin{theorem}
Let $m,\,n \in \N$, let $\Omega \subseteq \R^n$ be a domain that has the regular, uniform dilatation property and let $1 \leq p < \infty$.
Then the space $C^\infty_c(\bar{\Omega})$ is dense in $\widehat{W}^{m, p}(\Omega)$.
\end{theorem}
\begin{proof}
Let $x \in \R^n$ and let $r > 0$ such that $\varphi_k(\Omega \cap R_r(x),\,x) = \Omega \cap R_{k r}(x)$ for all $k \in \N$
and such that $\Omega \cap R_r(x) \neq \varnothing$ is a domain that has the cone property
and such that $\Omega\,\cap\,B_{k r}(x) \neq \varnothing$ is a domain with a locally Lipschitz continuous boundary for all $k \in \N$.
Let $\chi \in C^\infty(\R,\,[0,\,1])$ such that $\chi|_{(- \infty, 5 r / 8]} \equiv 1$ and $\chi|_{[7 r / 8, \infty)} \equiv 0$.
For $k \in \N$ let $\chi_k \in C^\infty_c(\R^n)$ be defined as $\chi_k(y) := \chi(k |y - x|)$ for $y \in \R^n$.
Then we have $\mbox{supp}\,\chi_k \subseteq B_{k r}(x)$ and $\mbox{supp}\,\nabla \chi_k \subseteq R_{k r}(x)$
as well as $\chi_k|_{\bar{B}_{k r / 2}(x)} \equiv 1$ for all $k \in \N$.
Moreover, there exists a constant $M > 0$ such that $\|\partial^\alpha \chi_k\|_{L^\infty(\R^n)} \leq k^{|\alpha|} M$ for all $|\alpha| \leq m$ and all $k \in \N$.
\pagebreak

Now, we fix $u \in \widehat{W}^{m, p}(\Omega)$.
For $k \in \N$ we then have $u|_{\Omega\,\cap\,B_{k r}(x)} \in \widehat{W}^{m, p}(\Omega \cap B_{k r}(x))$
and $\widehat{W}^{m, p}(\Omega \cap B_{k r}(x)) = W^{m, p}(\Omega \cap B_{k r}(x))$ algebraically,
since $\Omega \cap B_{k r}(x)$ has a locally Lipschitz continuous boundary, for all $k \in \N$; cf.~\cite[Remark~II.6.1]{Galdi:Navier-Stokes}.
Hence,
\begin{equation*}
	u_k := (u - p_{m - 1}[u,\,\Omega \cap B_{k r}(x)]) \chi_k \in \widehat{W}^{m, p}(\Omega)
\end{equation*}
is well-defined with $u_k|_{\Omega\,\cap\,B_{k r}(x)} \in W^{m, p}(\Omega\,\cap\,B_{k r}(x))$ for all $k \in \N$.
By construction we have
\begin{equation*}
	\|\partial^\alpha u_k\|_{L^p(\Omega)} = \|\partial^\alpha u_k\|_{L^p(\Omega\,\cap\,B_{k r}(x))} \leq \|u_k\|_{W^{m, p}(\Omega\,\cap\,B_{k r}(x))} < \infty,
		\qquad |\alpha| \leq m,
\end{equation*}
which shows that $u_k \in W^{m, p}(\Omega)$ for all $k \in \N$.
Now, we set $q_k := p_{m - 1}[u,\,\Omega \cap B_{k r}(x)]$ for $k \in \N$.
We fix $\alpha \in \N^n_0$ with $|\alpha| = m$ and observe that $\partial^\alpha q_k \equiv 0$ for all $k \in \N$.
Using Lemma~\ref{Uniform-Dilatation-Domain-Poincare-Higher-Order} we obtain
\begin{equation*}
	\begin{array}{l}
		\|\partial^\alpha(u_k - u)\|_{L^p(\Omega)}
			\leq \|(1 - \chi_k) \cdot \partial^\alpha u\|_{L^p(\Omega)}
				+ {\displaystyle \!\! \sum_{0 < \beta \leq \alpha} \!\! {\alpha \choose \beta} \|\partial^{\alpha - \beta} (u - q_k) \cdot \partial^\beta \chi_k\|_{L^p(\Omega)}} \\[2.0em]
			\qquad \leq \|\partial^\alpha u\|_{L^p(\Omega \setminus \bar{B}_{k r / 2}(x))}
				+ {\displaystyle \!\! \sum_{0 < \beta \leq \alpha} \!\! {\alpha \choose \beta} \|\partial^{\alpha - \beta} (u - q_k) \cdot \partial^\beta \chi_k\|_{L^p(\Omega\,\cap\,R_{k r}(x))}} \\[2.0em]
			\qquad \leq \|\partial^\alpha u\|_{L^p(\Omega \setminus \bar{B}_{k r / 2}(x))}
				+ {\displaystyle \!\! \sum_{0 < \beta \leq \alpha} \!\! {\alpha \choose \beta} k^{m - |\alpha - \beta|} C \|\nabla^m u\|_{L^p(\Omega\,\cap\,R_{k r}(x))} \cdot k^{|\beta|} M} \\[2.0em]
			\qquad \leq \|\partial^\alpha u\|_{L^p(\Omega \setminus \bar{B}_{k r / 2}(x))}
				+ {\displaystyle C M \!\! \sum_{0 < \beta \leq \alpha} \!\! {\alpha \choose \beta} \|\nabla^m u\|_{L^p(\Omega \setminus \bar{B}_{k r / 2}(x))}} \rightarrow 0 \quad \textrm{as}\ k \rightarrow \infty,
	\end{array}
\end{equation*}
since $\partial^\beta u \in L_p(\Omega)$ for all $\beta \in \N^n_0$ with $|\beta| = m$.
This shows that $u_k \rightarrow u$ in $\widehat{W}^{1, p}(\Omega)$ as $k \rightarrow \infty$.

Finally, since $\Omega$ has the segment property, for every $k \in \N$
there exists $v_k \in C^\infty_c(\bar{\Omega})$ such that $\|v_k - u_k\|_{W^{m, p}(\Omega)} < 2^{- k}$.
Therefore, we obtain
\begin{equation*}
	\|v_k - u\|_{\widehat{W}^{m, p}(\Omega)}
		\leq \|v_k - u_k\|_{W^{m, p}(\Omega)} + \|u_k - u\|_{\widehat{W}^{m, p}(\Omega)} \rightarrow 0 \quad \textrm{as}\ k \rightarrow \infty
\end{equation*}
and the proof is complete.
\end{proof}

\begin{corollary}
\label{Homogeneous-Density-Wedge}
Let $G \subseteq \R^2$ be the wedge domain with opening angle $\theta_0 \in (0, 2 \pi)$ defined as in (\ref{def_wedge}) 
and let $m \in \N$ and $1 \leq p < \infty$.
Then the space $C^\infty_c(\bar{G})$ is dense in $\widehat{W}^{m, p}(G)$.
\end{corollary}

\bibliographystyle{plain}
\bibliography{references}


\end{document}